\documentclass{article}
\usepackage{amssymb, amsthm, fancyhdr,amsmath,enumerate,graphicx,multicol}
\usepackage{subcaption}
\usepackage[inline]{enumitem}
\usepackage[numbers]{natbib}
\usepackage{color}
\usepackage{mathrsfs}
\usepackage{bbm}
\usepackage{appendix}
\usepackage[colorlinks=true,linkcolor=blue]{hyperref}
\usepackage{tcolorbox}
\usepackage{mathtools} 
\mathtoolsset{showonlyrefs=true}  
\usepackage[left=1.2in,top=1in,right=1.2in,bottom=1in]{geometry}
\usepackage{tikz-cd}

\usepackage{accents}

\DeclareMathOperator{\supp}{supp}

\DeclareMathOperator{\vol}{vol}

\DeclareMathOperator{\diam}{diam}

\DeclareMathOperator{\cpl}{Cpl}
\DeclareMathOperator{\cplbc}{Cpl_{bc}}

\DeclareMathOperator{\law}{Law}

\DeclareMathOperator{\var}{var}
\DeclareMathOperator{\mean}{m}

\begin{document}

\def\mf{\mathfrak{M}}
\def\a{\mathcal{A}}
\def\b{\mathcal{B}}
\def\c{\mathcal{C}}
\def\d{\mathscr{D}}
\def\e{\mathcal{E}}
\def\f{\mathcal{F}}
\def\g{\mathcal{G}}
\def\h{\mathcal{H}}
\def\i{\mathcal{I}}
\def\m{\mathcal{M}}
\def\n{\mathcal{N}}
\def\sp{\mathcal{P}}
\def\r{\mathcal{R}}
\def\s{\mathscr{S}}
\def\x{\mathcal{X}}
\def\y{\mathcal{Y}}
\def\z{\mathcal{Z}}
\def\t{\mathcal{T}}
\def\w{\mathcal{W}}

\def\A{\mathbb{A}}
\def\B{\mathbb{B}}
\def\C{\mathbb{C}}
\def\D{\mathbb{D}}
\def\E{\mathbb{E}}
\def\F{\mathbb{F}}
\def\H{\mathbb{H}}
\def\T{\mathbb{T}}
\def\L{\mathcal{L}}
\def\K{\mathbb{K}}
\def\N{\mathbb{N}}
\def\T{\mathbb{T}}
\def\P{\mathbb{P}}
\def\Q{\mathbb{Q}}
\def\R{\mathbb{R}}
\def\W{\mathbb{W}}
\def\Z{\mathbb{Z}}

\def\wstar{\overset{*}{\rightharpoonup}}
\def\nab{\nabla}
\def\dt{\partial_t}
\def\hal{\frac{1}{2}}
\def\ep{\varepsilon}
\def\vchi{\text{\large{$\chi$}}}
\def\ls{\lesssim}
\def\gs{\gtrsim}
\def\p{\partial}

\def\vn{\vol_n}
\def\om{\mu^\ast}

\def\lb{\log_b}

\def\wks{\buildrel\ast\over\rightharpoonup}
\def\wk{\rightharpoonup}

\def\fxx{\F\mspace{-3mu}\jump{\X}}
\def\fl{\F\mspace{-.5mu}(\mspace{-2mu}(\X)\mspace{-2mu})}
\def\afk{\mathbb{A}(\mathbb{K},\mathbb{F})}
\def\cf{\mathbb{C}(\F)}
\def\Re{\text{Re}}
\def\Im{\text{Im}}

\def\af{\mathfrak{A}}
\def\hf{\mathfrak{H}}
\def\ff{\mathfrak{F}}
\def\lf{\mathfrak{L}}
\def\bf{\mathfrak{B}}
\def\mf{\mathfrak{M}}
\def\pf{\mathfrak{P}}
\def\ep{\varepsilon}
\def\qed{{\hfill $\Box$ \bigskip}}
\def\st{\;\vert\;}
\def\vchi{\text{\large{$\chi$}}}

\newlist{anumerate}{enumerate}{1}
\setlist[anumerate,1]{label=(\alph*)}

\newcommand{\abs}[1]{\left\vert#1\right\vert}
\newcommand{\norm}[1]{\left\Vert#1\right\Vert}
\newcommand{\csubset}{\subset\subset}
\newcommand{\ind}[1]{\mathbbm{1}_{#1}}
\newcommand{\qnorm}[1]{\left \vert \mspace{-1.8mu} \left\vert
\mspace{-1.8mu} \left \lvert #1 \right \vert \mspace{-1.8mu} \right\vert
\mspace{-1.8mu} \right\vert}

\newcommand{\mhat}[1]{\accentset{\medianhat}{#1}}
\newcommand{\medianhat}{\smash{\raisebox{-1.08ex}{\scalebox{0.8}{$\widehat{\phantom{x}}$}}}} %

\newcommand{\br}[1]{\langle #1 \rangle}
\newcommand{\ns}[1]{\norm{#1}^2}
\newcommand{\ip}[1]{\left(#1 \right)}
 
\newcommand{\jump}[1]{\left\llbracket #1 \right\rrbracket }

\newtheorem{thm}{Theorem}
\newtheorem{cor}[thm]{Corollary}
\newtheorem{df}[thm]{Definition}
\newtheorem{assume}[thm]{Assumption}
\newtheorem{prop}[thm]{Proposition}
\newtheorem{lem}[thm]{Lemma}
\newtheorem{ex}[thm]{Example}
\newtheorem*{nota}{Notation}
\theoremstyle{definition}
\newtheorem{rmk}[thm]{Remark}

\title{The fast rate of convergence of the smooth adapted Wasserstein distance}
\author{Martin Larsson\thanks{Carnegie Mellon University, Department of Mathematical Sciences, \href{mailto:larsson@cmu.edu}{larsson@cmu.edu}} \and Jonghwa Park\thanks{Carnegie Mellon University, Department of Mathematical Sciences, \href{mailto:jonghwap@andrew.cmu.edu}{jonghwap@andrew.cmu.edu}} \and Johannes Wiesel\thanks{University of Copenhagen, Department of Mathematics, \href{mailto:wiesel@math.ku.dk}{wiesel@math.ku.dk}}}

\date{}

\maketitle

\begin{abstract}
Estimating a $d$-dimensional distribution $\mu$ by the empirical measure $\mhat{\mu}_n$ of its samples is an important task in probability theory, statistics and machine learning. It is well known  that $\E[\w_p(\mhat{\mu}_n, \mu)]\lesssim n^{-1/d}$ for $d>2p$, where $\w_p$ denotes the $p$-Wasserstein metric. An effective tool to combat this curse of dimensionality is the smooth Wasserstein distance $\w^{(\sigma)}_p$, which measures the distance between two probability measures after having convolved them with isotropic Gaussian noise $\mathcal{N}(0,\sigma^2\text{I})$. In this paper we apply this smoothing technique to the adapted Wasserstein distance. We show that the smooth adapted Wasserstein distance $\a\w_p^{(\sigma)}$ achieves the fast rate of convergence $\E[\a\w_p^{(\sigma)}(\mhat{\mu}_n, \mu)]\lesssim n^{-1/2}$, if $\mu$ is subgaussian. This result follows from the surprising fact, that any subgaussian measure $\mu$ convolved with a Gaussian distribution has locally Lipschitz kernels.

\end{abstract}

\section{Introduction}\label{sec:intro}
Let the Borel probability measures $\mu$ and $\nu$ be the laws of two stochastic processes $X=(X_t)_{t=1}^T$ and $Y=(Y_t)_{t=1}^T$ on the path space $(\R^d)^T$, where $d\ge 1$ and $T\ge 2$. Let furthermore $\sp_p((\R^d)^T)$ denote the set of all Borel probability measures on $(\R^d)^T$ with finite $p$-moments, where $1\le p<\infty$ is fixed. The weak topology on $\sp_p((\R^d)^T)$ is metrized by the \emph{$p$-Wasserstein distance}
\begin{align}
    \w_p(\mu, \nu)
    :=\left(\inf_{\pi \in \cpl(\mu, \nu)}\int \abs{x-y}^p \pi(dx, dy)\right)^{1/p}.\label{eq:wass}
\end{align}
Here, $\abs{\cdot}$ denotes the $\ell^2$-norm on $(\R^d)^T$ and $\cpl(\mu, \nu)$ is the set of all couplings of $\mu$ and $\nu$; see \cite{villani2009optimal, villani2021topics, santambrogio2015optimal} for a general overview of optimal transport theory and  the Wasserstein distance.

For computational as well as estimation purposes, $\mu$ is often approximated by its empirical distribution $\mhat{\mu}_n:=\frac{1}{n}\sum_{j=1}^n \delta_{X^{(j)}}$ where $X^{(1)}, \ldots, X^{(n)}$ are i.i.d samples from $\mu$. By the Glivenko--Cantelli theorem, $\mhat{\mu}_n$ converges weakly to $\mu$ almost surely as the sample size $n$ approaches infinity. As a consequence, $\w_p(\mhat{\mu}_n, \mu)$ vanishes with probability one. However, this convergence is severely impeded by an exponential dependence on the dimension $dT$ of the path space, posing a challenge for computational efficiency. In fact, \cite{fournier2015rate} shows the sharp \emph{curse of dimensionality (COD)} convergence rates $\E[\w_p(\mhat{\mu}_n, \mu)]\le C n^{-1/(dT)}$ whenever $dT>2p$ and $\mu$ is supported on $([0,1]^d)^T$.

In order to improve these convergence rates, the \emph{smooth $p$-Wasserstein distance} $\w^{(\sigma)}_p$ was recently studied \cite{goldfeld2020gaussian, goldfeld2020convergence, goldfeld2020asymptotic, nietert2021smooth, sadhu2021limit, goldfeld2024limit, goldfeld2024statistical}.
\begin{df}[Smooth Wasserstein distance] Let $1\le p<\infty$ and $\mu, \nu\in \sp_p((\R^d)^T)$. The smooth $p$-Wasserstein distance between $\mu$ and $\nu$ with smoothing parameter $\sigma>0$ is defined as
\begin{align}\label{eq:smoothw}
    \w^{(\sigma)}_p(\mu, \nu)
    :=\w_p(\mu\ast \n_{\sigma}, \nu\ast \n_{\sigma}).
\end{align}
Here $\ast$ denotes the convolution operator and $\n_{\sigma}=\n(0, \sigma^2 \text{I}_{dT})$ is an isotropic Gaussian measure on $(\R^d)^T$.
\end{df}
Smoothing $\w_p$ in this way leads to several interesting results. In particular, the expected $\w^{(\sigma)}_p$-distance between $\mhat{\mu}_n$ and $\mu$ exhibits dimension-free convergence rates; this clearly improves upon the classical Wasserstein convergence rates discussed above. The following list gives a general overview of known results for $\w^{(\sigma)}_p$:
\begin{enumerate}[label=(\arabic*)]
    \item\label{ov:topology} \textbf{Topological equivalence.} The two metrics $\w^{(\sigma)}_p$ and $\w_p$ generate the same topology on $\sp_p((\R^d)^T)$. See \cite{goldfeld2020gaussian, nietert2021smooth}.
    \item\label{ov:proxy} \textbf{Stability.} As $\sigma\to 0$, $\w^{(\sigma)}_p\to \w_p$ and the metric $\w^{(\sigma)}_p$ serves as a proxy for the original metric $\w_p$. In particular, \cite{nietert2021smooth} shows that for some $C>0$,
    \begin{align}
        \w^{(\sigma)}_p(\mu, \nu)
        \le \w_p(\mu, \nu)
        \le \w^{(\sigma)}_p(\mu, \nu)+C\sigma.
    \end{align}
    \item\label{ov:slowrate} \textbf{Slow rate.} In \cite{goldfeld2020convergence, nietert2021smooth} it is shown that $\E[\w^{(\sigma)}_p(\mhat{\mu}_n, \mu)^p]\le Cn^{-1/2}$ holds when $\mu$ has a finite $q$-moment for some $q>2(dT+p)$. This result implies the \emph{slow rate} of convergence, i.e., for some $C>0$,
    \begin{align}
        \E[\w^{(\sigma)}_p(\mhat{\mu}_n, \mu)]\le Cn^{-1/(2p)}.
    \end{align}
    \item\label{ov:fastrate} \textbf{Fast rate.} If $\int e^{\beta \abs{x}^2}\mu(dx)<\infty$ for some $\beta>(p-1)/\sigma^2$ (requiring $\mu$ to be subgaussian), this can be improved to the \emph{fast rate} of convergence, i.e., for some $C>0$,
    \begin{align}
        \E[\w^{(\sigma)}_p(\mhat{\mu}_n, \mu)]\le C n^{-1/2};
    \end{align}
    see \cite{nietert2021smooth, block2022rate}. Moreover, \cite{goldfeld2020convergence, block2022rate} together with \cite{goldfeld2024limit} suggest that subgaussianity is necessary for the fast rate to hold.
    \item\label{ov:lim} \textbf{Distributional limit.} Under the same assumption as in \ref{ov:fastrate}, $\sqrt{n}\w^{(\sigma)}_p(\mhat{\mu}_n, \mu)$  converges weakly to a supremum of a tight Gaussian process. Details can be found in \cite{goldfeld2020asymptotic, sadhu2021limit} for $p=1$ and \cite{goldfeld2024limit, goldfeld2024statistical} for $p>1$.
\end{enumerate}

Taken together, the first two items show that the smoothed distance $\w^{(\sigma)}_p$ provides a good proxy for $\w_p$. For small values of $\sigma$, the smoothed distance remains close to the original Wasserstein distance. The third and fourth items highlight the improved statistical convergence rates achieved by the smoothed distance $\w^{(\sigma)}_p$, which makes it particularly appealing for statistical applications.

When the dimension is sufficiently small, the slow rate $n^{-1/(2p)}$ coincides with the classical Wasserstein rates for $\E[\w_p(\widehat{\mu}_n, \mu)]$; see \cite{fournier2015rate}. Smoothing ensures that the rate $n^{-1/(2p)}$ remain valid even in high dimensions.

When $p>1$, the fast rate~$n^{-1/2}$ implies strictly faster convergence than the slow rate~$n^{-1/(2p)}$. Going back to the seminal work of Ajtai--Koml\'os--Tusn\'ady \cite{ajtai1984optimal}, it is well-known that the fast rate~$n^{-1/2}$ is achievable only if the support of $\mu$ is regular enough or topologically connected. In our setting, convolving $\mu$ with a Gaussian measure inherently ensures this connectivity by spreading mass across the entire space, and thereby facilitating the rate $n^{-1/2}$. While establishing the slow rate~\ref{ov:slowrate} under minimal assumptions is of independent interest, the distributional limit~\ref{ov:lim} suggests that the fast rate~\ref{ov:fastrate} is sharp in general. As a simple example, take $\mu=\frac{1}{2}(\delta_{a}+\delta_{b})$ for $a\neq b$. Similar to \cite[Example (a) on page 2]{fournier2015rate}, we have $\E[\w_p^{(\sigma)}(\mhat{\mu}_n, \mu)]\ge Cn^{-1/2}$ (see Appendix~\ref{appendix:sharp} for details).

While smoothing offers statistical advantages, it also introduces certain computational challenges. The convolution approach yields on optimal transport problem of Gaussian mixtures. To the best of our knowledge, efficient algorithms that exploit this mixture structure are currently unknown. One possible approach is to approximate the convolution by sampling from the Gaussian kernel, although this may partially offset the statistical advantages of the smoothed distance.

In this paper, we extend the smoothing technique for $\w_p^{(\sigma)}$ to the \textit{adapted Wasserstein distance} $\a\w_p$ and study the statistical properties of its smoothed counterpart $ \a\w_p^{(\sigma)}$, the \textit{smooth adapted Wasserstein distance}. The adapted Wasserstein distance  was introduced to address the following issue: 
\begin{quote}
For many time-dependent operators $F:\sp_p((\R^d)^T)\to \R$, $\liminf_{n\to \infty} |F(\mu_n)-F(\mu)|>0$ even if $\lim_{n\to \infty} \w_p(\mu_n, \mu)=0$.
\end{quote}
Examples of such operators $F$ are the value functions of optimal stopping problems, the Doob decomposition, superhedging problems, utility maximization, stochastic programming and risk measurements \cite{pflug2012distance, pflug2014multistage, glanzer2019incorporating, acciaio2020causal, backhoff2020adapted}; these commonly account for the time-structure of $\mu$ and $\nu$. As we describe in more detail below, $\a\w_p$ generates the coarsest topology, which makes such operators continuous \cite{backhoff2020all}. Having introduced $\a\w_p^{(\sigma)}$, the main result of this paper, presented in Theorem~\ref{thm:fast_rate}, can be summarized as follows:

\begin{tcolorbox}
If $\mu$ is subgaussian, then $\E[\a\w^{(\sigma)}_p(\mu, \mhat{\mu}_n)]\lesssim 1/\sqrt{n}$, i.e., the fast rate of convergence holds for the smooth adapted Wasserstein distance.
\end{tcolorbox}

Before discussing Theorem~\ref{thm:fast_rate} in Section~\ref{sec:main} in more detail, we first recall basic facts about $\a\w_p$ in Section~\ref{sec:adapted distances} and review existing results on finite-sample guarantees for $\a\w_p$ and $\a\w^{(\sigma)}_p$ in Section~\ref{sec:approx adapted} and Section~\ref{sec:prior smooth aw}, respectively.

\subsection{Adapted distances}\label{sec:adapted distances}

To illustrate the fact that the usual Wasserstein distance $\w_p$ is inadequate for time-dependent optimization problems, take $T=2, d=1$ and consider the laws $ \mu=\frac{1}{2}\delta_{(0,1)}+\frac{1}{2}\delta_{(0, -1)}$ and $ \mu_{\ep}=\frac{1}{2}\delta_{(\ep, 1)}+\frac{1}{2}\delta_{(-\ep, -1)}$ for $\ep>0$. Figure~\ref{fig:example_path} illustrates their sample paths. It is evident that $\mu_{\ep}$ is close to $\mu$ in Wasserstein distance for small $\ep$; in fact, $\w_p(\mu_{\ep}, \mu)=\ep$. However, the process $X^{\ep}\sim \mu^\epsilon$ differs significantly from $X\sim \mu$ as its values at time $2$ are already determined at time $1$. To see this, note that for $X=(X_1, X_2)\sim \mu$ and $X^{\ep}=(X^{\ep}_1, X^{\ep}_2)\sim \mu_{\ep}$,
\begin{align}
    \law(X_2 | X_1)=\left(\frac{1}{2}\delta_1+\frac{1}{2}\delta_{-1}\right)\ind{\{X_1=0\}},\quad
    \law(X^{\ep}_2 \,|\, X^{\ep}_1)
    =\delta_1 \ind{\{X^{\ep}_1=\ep\}}+\delta_{-1} \ind{\{X^{\ep}_1=-\ep\}}.
\end{align}
As a consequence, the values of utility maximization problems and optimal stopping problems for $\mu^\epsilon$ do not converge to the corresponding values for $\mu$.
\begin{figure}[ht]
    \centering
    \includegraphics[width = 8.75cm, height = 2.5cm]{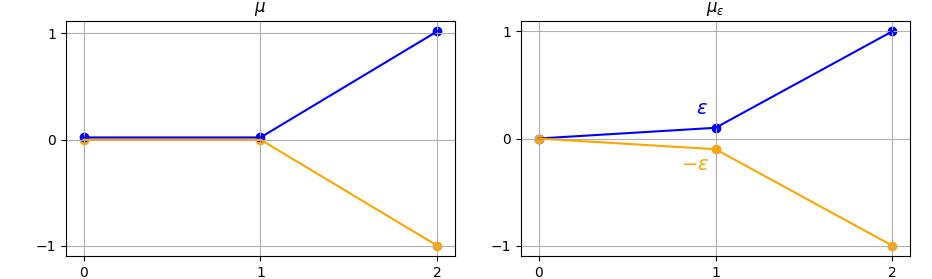}
    \caption{$\mu=\frac{1}{2}\delta_{(0,1)}+\frac{1}{2}\delta_{(0,-1)}$ on the left and $\mu_{\ep}=\frac{1}{2}\delta_{(\ep, 1)}+\frac{1}{2}\delta_{(-\ep, -1)}$ on the right.} 
    \label{fig:example_path}
\end{figure}

To take the flow of information formalized through the natural filtration of a stochastic process into account, several variants of the weak topology have been independently developed by different communities, e.g., \cite{kantorovich1942translocation, Al81, hoover1984adapted, ruschendorf1985wasserstein, hellwig1996sequential, lassalle2018causal, acciaio2019extended, bonnier2023adapted}. Notably, \cite{backhoff2020all} demonstrates that all these seemingly different variants of topologies coincide and are generated by the adapted Wasserstein distance. In order to define it formally, we first introduce the notion of \emph{bicausal couplings}.

\begin{df}[Bicausal coupling] Let $\mu$ and $\nu$ be two probability measures on $(\R^d)^T$. A coupling $\pi \in \cpl(\mu, \nu)$ is bicausal if for $(X, Y)\sim \pi$ and $t\in \{1,2,\ldots, T-1\}$,
\begin{align}
    (Y_1, \ldots, Y_t) \text{ and } (X_{t+1}, \ldots, X_T) \text{ are conditionally independent given } X_1, \ldots, X_t,
\end{align}
and
\begin{align}
    (X_1, \ldots, X_t) \text{ and } (Y_{t+1}, \ldots, Y_T) \text{ are conditionally independent given } Y_1, \ldots, Y_t.
\end{align}
The set of all bicausal couplings between $\mu$ and $\nu$ is denoted by $\cplbc(\mu, \nu)$.
\end{df}

\begin{df}[The adapted Wasserstein distance]\label{df:aw} Let $1\le p<\infty$ and $\mu, \nu\in \sp_p((\R^d)^T)$. The adapted $p$-Wasserstein distance between $\mu$ and $\nu$ is defined as
\begin{align}\label{eq:adw}
    \a\w_p(\mu, \nu)
    :=\left(\inf_{\pi \in \cplbc(\mu, \nu) }\int \sum_{t=1}^T \abs{x_t-y_t}^p \pi(dx, dy)\right)^{1/p}.
\end{align}
Similarly to \eqref{eq:smoothw}, the smooth adapted $p$-Wasserstein distance between $\mu$ and $\nu$ with smoothing parameter $\sigma>0$ is defined as
\begin{align}
    \a\w^{(\sigma)}_p(\mu, \nu)
    :=\a\w_p(\mu\ast \n_{\sigma}, \nu\ast \n_{\sigma}).
\end{align}
\end{df}
As mentioned above, the adapted Wasserstein distance induces the coarsest topology that makes optimal stopping problems continuous \cite{backhoff2020all}. This topology is finer than the weak topology.

\subsection{Approximations in adapted distances}\label{sec:approx adapted}
Unlike the Wasserstein case, it is well-known (e.g., \cite{pflug2016empirical, backhoff2022estimating}) that the empirical measure $\mhat{\mu}_n$ does not converge to $\mu$ in adapted Wasserstein distance as $n\to \infty$. To ensure approximation of $\mu$ by empirical data in adapted Wasserstein sense, \cite{backhoff2022estimating} devise the so-called \emph{adapted empirical measure} $\boldsymbol{\mhat{\mu}_n}$ as an alternative to the empirical measure $\mhat{\mu}_n$. It is defined as a projection of $\mhat{\mu}_n$ onto a refining grid and satisfies $\lim_{n\to \infty} \a\w_p(\boldsymbol{\mhat{\mu}_n},\mu)=0$. However the convergence rate obtained for $\boldsymbol{\mhat{\mu}_n}$ is essentially the same as the Wasserstein COD rates for $\E[\w_p(\mhat{\mu}_n, \mu)]$. In fact, it was shown in \cite{backhoff2022estimating, acciaio2024convergence} that $\E[\a\w_1(\boldsymbol{\mhat{\mu}_n}, \mu)]\le C n^{-1/(dT)}$ for some $C>0$ when $d\ge 3$ and $\mu$ is a probability measure on $([0,1]^d)^T$ that has Lipschitz kernels. As $\w_p\le C\a\w_p$ for some constant $C$ that depends only on $d, T$, the rates for $\E[\a\w_1(\boldsymbol{\mhat{\mu}_n}, \mu)]$ are sharp.

Motivated by the non-adapted counterpart $\w^{(\sigma)}_p$, smoothing techniques have been introduced to achieve dimension-free adapted Wasserstein approximations. One of the earliest works in this direction is \cite{pflug2016empirical}, which states that $\a\w_p(\mu, \mhat{\mu}_n\ast \eta_{\sigma_n})$ converges to zero in probability under arguably restrictive assumptions on $\mu$. Here $\eta_{\sigma_n}$ are non-Gaussian smoothing kernels converging weakly to the Dirac distribution $\delta_0$ for $\sigma_n \to 0$. Precise statements can be found in \cite[Theorem $4$]{pflug2016empirical}. On the contrary, we keep $\sigma>0$ fixed throughout this work.

\subsection{Prior results for the smooth adapted Wasserstein distance}\label{sec:prior smooth aw}

Paralleling our list for $\w_p^{(\sigma)}$ above, let us now provide an overview of known results for $\a\w_p^{(\sigma)}$. It seems natural to conjecture that items (1)-(5) still hold when replacing $\w_p$ by $\a\w_p$ and $\w_p^{(\sigma)}$ by $\a\w_p^{(\sigma)}$. Perhaps surprisingly, it turns out that this is not the case, as already the first item on the list fails. Here is the corresponding list of facts for the adapted Wasserstein distance:

\begin{enumerate}[label=(\arabic*)]
\item\label{ov:aw topology} \textbf{Topology.} The two metrics $\a\w^{(\sigma)}_p$ and $\w_p$ generate the same topology on on $\sp_p((\R^d)^T)$ \cite{blanchet2024bounding}. Note that this is \emph{not} the same topology as the one generated by $\a\w_p$.

\item\label{ov:aw stability} \textbf{Stability.} For small values of $\sigma$, $\a\w^{(\sigma)}_p$ closely approximates $\a\w_p$. More specifically, \cite{blanchet2024bounding}~establishes that
\begin{align} \label{eq:old}
    \abs{\a\w^{(\sigma)}_p(\mu, \nu)-\a\w_p(\mu, \nu)}\le C(\omega_{\mu}(\sigma)+\omega_{\nu}(\sigma))
\end{align}
where $\omega_{\mu}$ quantifies the regularity of the kernels of $\mu$. The precise definition of $\omega_{\mu}$ is given in \cite[Theorem $26$]{blanchet2024bounding}. Let us emphasize that $\lim_{\sigma\to 0}\omega_{\mu}(\sigma)=0$. In particular, if $\mu$ has Lipschitz kernels, $\omega_{\mu}(\sigma)\le C \sigma$ in \eqref{eq:old}.

\item\label{ov:aw slow rate} \textbf{Slow rate.} For $p=1$, \cite{hou2024convergence} shows that $\E[\a\w^{(\sigma)}_1(\mhat{\mu}_n, \mu)]\le Cn^{-1/2}$ for compactly supported $\mu$.
For general $p\ge 1$, \cite{blanchet2024bounding} shows that $\E[\a\w_p^{(\sigma)}(\mhat{\mu}_n, \mu)^p]\le Cn^{-1/2+p/(2q)}$ if $\int \abs{x}^q\mu(dx)<\infty$ for some $q>p\vee (dT+2)$. In particular, if $\mu$ has a compact support, we can take $q$ arbitrarily large, resulting in the bound $\E[\a\w^{(\sigma)}_p(\mhat{\mu}_n, \mu)^p]\le C n^{-1/2+\ep}$ for arbitrarily small $\ep>0$ and a constant $C>0$ depending on $\ep$. This recovers the slow rate \ref{ov:slowrate} for $\a\w^{(\sigma)}_p$, up to a loss of $\ep$.

\item\label{ov:aw fast rate} \textbf{Fast rate.} To the best of our knowledge, there is no existing result for the fast rate for $\a\w^{(\sigma)}_p$. Our paper fills this gap.

\item\label{ov:aw distributional limit} \textbf{Distributional limit.} To the best of our knowledge, there is no result for the limiting distribution of $\sqrt{n}\a\w^{(\sigma)}_p(\mhat{\mu}_n, \mu)$. We plan to address this question in future research.

\end{enumerate}

While the first result suggests that the smooth variant does not preserve the topological structure induced by the metric $\a\w_p$, the stability bound in the second item~\ref{ov:aw stability} demonstrates that $\a\w^{(\sigma)}_p$ still provides a quantitative proxy for $\a\w_p$. In particular, when the smoothing parameter $\sigma$ is small, the gap between $\a\w_p$ and $\a\w^{(\sigma)}_p$ can be controlled quantitatively. From a statistical point of view, this makes $\a\w^{(\sigma)}_p$ a natural surrogate for $\a\w_p$, in the same spirit that $\w^{(\sigma)}_p$ is often used as a surrogate for the classical Wasserstein distance $\w_p$. \\
Furthermore, the estimator $\widehat{\mu}^{\sigma}_n$ is consistent for $\a\w_p$ and its rate of convergence can be quantified via the rates for $\a\w^{(\sigma)}_p$, as 
\begin{align}
    \mathbb{E}[\a\w_p(\widehat{\mu}^{\sigma}_n, \mu)]
    &\le \mathbb{E}[\a\w^{(\sigma)}_p(\widehat{\mu}_n, \mu)]
    +\a\w_p(\mu^{\sigma}, \mu)\\
    &\le C_1 n^{-1/2} + C_2 \omega_\mu(\sigma).
\end{align}
For a fixed value of $\sigma$, the number of samples needed to approximate $\mu$ at a precision of $\epsilon$ is thus of the dimension free order $n^{1/2}/\epsilon$. \\

Regarding the fast rate for $\a\w^{(\sigma)}_p$, the closest paper we could find is \cite{blanchet2024empirical}, where the authors introduce the so-called \textit{smoothed empirical martingale projection distance} $\text{MPD}^{\ast \xi}(\mhat{\mu}_n, p)$. It is defined as the minimal $\a\w^p_p$-value between $\mhat{\mu}_n\ast \xi$ for a smoothing distribution $\xi$ and the set of martingale measures  (i.e., the set of measures $\nu$ satisfying $\E[X_2|X_1]=X_1$ if  $(X_1, X_2)\sim \nu$). $\text{MPD}^{\ast \xi}(\mhat{\mu}_n, p)$ is used to construct a test for the martingale property of $\mu$. When $\mu$ is a martingale measure, \cite{blanchet2024empirical} shows that $n^{p/2}\text{MPD}^{\ast \xi}(\mhat{\mu}_n, p)$ has a weak limit. Although $\xi$ is not necessarily Gaussian and their setting differs from ours, it reflects the spirit of the fast rate we aim to explore. This is why we will examine this example in more detail in Section \ref{sec:appl}.

\subsection{Main results}\label{sec:main}

We are now in a position to state our main result: the fast rate for the smooth adapted $p$-Wasserstein distance and subgaussian measures $\mu$.

\begin{thm}[Fast rate]\label{thm:fast_rate}
Let $1<p<\infty$ and $\sigma>0$. Suppose that $\mu$ is a probability measure on $(\R^d)^T$, where $d\ge 1$ and $T\ge 2$, such that $\int e^{q\abs{x}^2/(2\sigma^2)}\mu(dx)<\infty$ for $q>8p(2p-1)(T+9)$. Then there exists a constant $C>0$ that depends only on $p, q, d, T, \sigma$ and $\int e^{q\abs{x}^2/(2\sigma^2)}\mu(dx)$ such that
    \begin{align}\label{eq:fast_rate}
        \E[\a\w^{(\sigma)}_p(\mhat{\mu}_n,\mu)]\le \frac{C}{\sqrt{n}}.
    \end{align}
\end{thm}

Using the fact that $\w_p\le C\a\w_p$ for some constant $C$ that depends only on $d, T$, Theorem \ref{thm:fast_rate} is sharp. Furthermore, recall from~\ref{ov:fastrate} for $\w^{(\sigma)}_p$ that the fast rate $n^{-1/2}$ essentially holds only if $\mu$ is subgaussian. Theorem~\ref{thm:fast_rate} shows that the converse is true.

\begin{rmk}[Dependence of $C$ on $\sigma$]
While the convergence rate $n^{-1/2}$ is dimension-free, the prefactor $C$ may grow exponentially with the dimension, especially when $\sigma$ is small. In the special case where $p=1$, \cite[Theorem $4.1$]{hou2024convergence} shows that for $\a\w^{(\sigma)}_1$, the dependence of $C$ on small values of $\sigma$ is $\sigma^{-dT/2}$, which becomes dominant in high dimensions. As pointed out in \cite[Remark $3$]{nietert2021smooth}, a similar phenomenon occurs for the non-adapted counterpart $\w^{(\sigma)}_p$. Indeed, when $p=1$, \cite[Proposition $1$]{goldfeld2020convergence} shows that the corresponding dependence also scales as $\sigma^{-dT/2}$ (for an underlying $dT$-dimensional space). For higher values of $p>1$, one may expect the dependence on $\sigma$ to deteriorate further. Our proof focuses on establishing the sharp rate $n^{-1/2}$ in the case $p>1$ and does not track this dependence explicitly. 
\end{rmk}

We detail the proof of Theorem~\ref{thm:fast_rate} in Section~\ref{sec:proof}. We also remark that the moment assumption $q>8p(2p-1)(T+9)$ can be relaxed. In fact, we will see in Section~\ref{sec:proof}, that \eqref{eq:fast_rate} holds if $q>\inf_{\beta} q^{*}(p, T, \beta)$, where the infimum is taken over all parameters $\beta$ in the set \eqref{assume:parameters}. We refer to \eqref{assume:q star} for a precise definition of $q^{*}(p, T, \beta)$.

The proof of Theorem \ref{thm:fast_rate} combines the dynamic programming principle for the adapted Wasserstein distance (see Proposition \ref{prop:dpp}) with the following rather surprising result, stating that \emph{any} compactly supported measure convolved with Gaussian noise automatically has Lipschitz kernels:

\begin{prop}[Smoothed measures have Lipschitz kernels; exact statement in Proposition \ref{prop:lipkernel}]
    Suppose $\mu$ is a compactly supported probability measure on $(\R^d)^T$. Then there exists a constant $C>0$ that depends only on $d, p, \sigma, \supp(\mu)$ such that
    \begin{align}
        \w_p((\mu\ast \n_{\sigma})_{x_{1:t}}, (\mu\ast \n_{\sigma})_{y_{1:t}})
        \le C\abs{x_{1:t}-y_{1:t}}
    \end{align}
    for all $x, y\in (\R^d)^T$ and $t\in \{1,2,\ldots, T-1\}$. Here, $x_{1:t}$ denotes the first $t$ coordinates of $x$ and $(\mu\ast \n_{\sigma})_{x_{1:t}}$ is a probability measure on $\R^d$ defined via
    \begin{align}
        (\mu\ast \n_{\sigma})_{x_{1:t}}(dx_{t+1})
        :=\P(X_{t+1}\in dx_{t+1} \,|\, X_1= x_1, \cdots, X_t=x_t)
    \end{align}
    for $X\sim \mu\ast\n_{\sigma}$.
\end{prop}

In Proposition \ref{prop:lipkernel} below we extend this result to subgaussian measures, showing that a smoothed subgaussian measure has \textit{locally} Lipschitz kernels. Since Lipschitz kernels naturally arise in many applications and are well-studied \cite{backhoff2022estimating, acciaio2024convergence, blanchet2024bounding, eder2019compactness}, we believe that Proposition \ref{prop:lipkernel} is of independent interest.

\subsection{Applications}\label{sec:appl}

We now present a concrete application of the fast rate \eqref{eq:fast_rate}, which is taken from \cite{blanchet2024empirical}. Consider a probability measure $\mu$ on $(\R^d)^2$. Recall that we call $\mu$ a martingale measure, if $\E[X_2|X_1]=X_1$ holds for $(X_1, X_2)\sim \mu$. Closely following \cite{blanchet2024empirical}, let us define the \textit{smoothed martingale projection distance (SMPD)} via
\begin{align}
    \text{SMPD}(\mu, p)
    :=\inf\big\{\a\w_p(\mu\ast \xi, \nu) \,|\, \nu \text{ is a martingale measure} \big\}.
\end{align}
Here we define $\xi$ as the law of $(Z_1, Z_1+Z_2)$, where $(Z_1, Z_2)$ are two independent Gaussian random variables with mean $0$ and covariance matrix $\text{I}_{d}$. $ \text{SMPD}(\mu, p)$ measures the $\a\w_p$-distance between $\mu\ast \xi$ and the space of martingale measures, and is used to test whether $\mu$ is a martingale measure. While \cite{blanchet2024empirical} offers an in-depth analysis of this distance, let us emphasize here that $\mu$ is a martingale if and only if $\text{SMPD}(\mu, p)=0$. In particular, to test the martingale hypothesis using i.i.d samples, we aim to bound the probability
\begin{align}
    \P\left(\abs{\text{SMPD}(\mu, p)-\text{SMPD}(\mhat{\mu}_n, p)}>\alpha\right)
\end{align}
for $\alpha>0$. By the triangle inequality,
\begin{align}
    \abs{\text{SMPD}(\mu, p)-\text{SMPD}(\mhat{\mu}_n, p)}
    \le \a\w_p(\mu\ast \xi, \mhat{\mu}_n\ast \xi).
\end{align}
Combined with the Markov inequality leads to
\begin{align}
    \P\left(\abs{\text{SMPD}(\mu, p)-\text{SMPD}(\mhat{\mu}_n, p)}>\alpha\right)
    \le \frac{\E[\a\w_p(\mu\ast \xi, \mhat{\mu}_n\ast \xi)]}{\alpha}.\label{eq:testmartingale}
\end{align}
The fast rate \eqref{eq:fast_rate} applied to the right hand side of \eqref{eq:testmartingale} establishes the convergence rate $n^{-1/2}$,  which is independent of the dimension $d$.

\subsection{Notation and preparations}\label{sec:nota}
We close this section by setting up notation in Section \ref{sec:nota}.  As mentioned above, $\abs{\cdot}$ is the Euclidean norm, and we denote the scalar (dot) product by $\cdot$. Throughout the paper, $T\ge 2$ is the number of time steps and $d\ge 1$ is the dimension of the state space. The H\"older conjugate of $p$ is denoted by $p'$, i.e., $1/p+1/p'=1$.

For any Borel set $A$ in Euclidean space, the set of all Borel probability measures on $A$ is denoted by $\sp(A)$. For $1\le p<\infty$, $\sp_p(A)$ is the set of all $\mu\in \sp(A)$ that have finite $p$-moments, i.e., $\int \abs{x}^p \mu(dx)<\infty$. For a measure $\mu$, we denote the pushforward measure of $\mu$ under a Borel function $T$ by $T_{\#}\mu$, i.e., $T_{\#}\mu(A)=\mu(\{x : T(x)\in A\})$ for all Borel sets $A$.

Given $X\sim \mu\in \sp((\R^d)^T)$, we denote the mean of $X$ by $\mean(\mu):=\int x \mu(dx)$. Also, the trace of the covariance matrix of $X$ is denoted by $\var(\mu):=\int \abs{x-\mean(\mu)}^2 \mu(dx)$. For $r>0$, we define $M_{r}(\mu):=\int \abs{x}^r \mu(dx)$ and $\e_{r}(\mu):=\int e^{r\abs{x}^2}\mu(dx)$.

For $x\in (\R^d)^T$ and $t\in \{1,2,\ldots, T\}$, we use the shorthand notation $x_t$ to denote the $t$-coordinate of $x$ and $x_{1:t}:=(x_1, \ldots, x_t)$. In particular, $x_{1:T}=x$. Similarly, given $\mu\in \sp((\R^d)^T)$, we write $\mu_{t}$ for the projection of $\mu$ onto the $t$-coordinate and $\mu_{1:t}$ for the projection of $\mu$ onto the first $t$-coordinates. Precisely speaking, $\mu_t$=$P^t_{\#}\mu$ and $\mu_{1:t}=P^{1:t}_{\#}\mu$ for $P^t(x)=x_t$ and $P^{1:t}(x)=x_{1:t}$.

Given $x\in (\R^d)^T$ and $\mu\in \sp((\R^d)^T)$, recall that a disintegration(or kernel) of $\mu$ is a measure $\mu_{x_{1:t}}\in \sp(\R^d)$, $t\in \{1,2,\ldots, T\}$ which is defined via $\mu_{x_{1:t}}(dx_{t+1})=\P(X_{t+1}\in dx_{t+1} \st X_{1:t}=x_{1:t})$ for $X\sim \mu\in \sp((\R^d)^T)$.

Given $\sigma>0$, we write $\varphi_{\sigma}:(\R^d)^T\to \R$ for the Gaussian density, i.e., $\varphi_{\sigma}(x)=(2\pi\sigma^2)^{-dT/2}e^{-\abs{x}^2/2}$. It is the density of the centered Gaussian measure on $(\R^d)^T$ with covariance matrix $\sigma^2 \text{I}_{dT}$, which is denoted by $\n_{\sigma}$. Given $\mu\in \sp((\R^d)^T)$, $\mu\ast \n_{\sigma}$ is a convolution of $\mu$ and $\n_{\sigma}$, i.e., $\mu\ast \n_{\sigma}(A)=\int \n_{\sigma}(A-x)\mu(dx)$ for all measurable $A$. Note that $\mu\ast \n_{\sigma}$ has density $x\mapsto \varphi_{\sigma}\ast \mu(x)=\int \varphi_{\sigma}(x-y)\mu(dy)$. We use the shorthand notation $\mu^{\sigma}:=\mu\ast \n_{\sigma}$. For $t\in \{1,2,\ldots, T\}$, we abuse notation and write $\varphi_{\sigma}(x_{1:t})=(2\pi\sigma^2)^{-dt/2}e^{-\abs{x_{1:t}}^2/2}$. Similarly, $\n_{\sigma}$ can mean a centered Gaussian measure on $(\R^d)^t$ with covariance matrix $\sigma^2\text{I}_{dt}$ depending on the context. Following the same reasoning, the density of $(P^{1:t}_{\#}\mu)^{\sigma}= (P^{1:t}_{\#}\mu)\ast \n_{\sigma}$ is denoted by $x_{1:t}\mapsto \varphi_{\sigma}\ast \mu(x_{1:t})= \int \varphi_{\sigma}(x_{1:t}-y_{1:t})\mu(dy_{1:t})$.

For a probability measure $\mu$ and i.i.d samples $X^{(1)}, X^{(2)}, \ldots, X^{(n)}$ of $\mu$, we define the empirical measure of $\mu$ via $\mhat{\mu}_n = \frac{1}{n}\sum_{j=1}^n \delta_{X^{(j)}}$. Note that this is a measure-valued random variable. Adopting the same notation as above, we write $\mhat{\mu}^{\sigma}_n$ for $\mhat{\mu}_n\ast \n_{\sigma}$.

Let $\mu\in \sp(A)$ for some Borel set $A$ in a Euclidean space. For $1\le p< \infty$, $L^p(\mu;\R^k)$ is the set of all functions $f:A\to \R^k$ such that $\norm{f}_{L^p(\mu;\R^k)}:=(\int_{A} \abs{f}^p d\mu)^{1/p}<\infty$. When $k=1$, we often write $L^p(\mu)\:=L^p(\mu;\R^k)$. We denote by $C_c^{\infty}(\R^N)$ the set of all smooth functions $h:\R^N\to \R$ with compact support. For the multi-index $\alpha=(\alpha_1, \ldots, \alpha_N)$ and $h\in \R^N\to \R$, we write $\abs{\alpha}=\sum_{j=1}^N \alpha_j$ and denote the $\alpha$-th derivative of $h$ by $\p^{\alpha}h(x)=\frac{\p^{\alpha_1}}{\p x_{1}^{\alpha_1}}\frac{\p^{\alpha_2}}{\p x_{2}^{\alpha_2}}\cdots \frac{\p^{\alpha_N}}{\p x_{N}^{\alpha_N}}h(x)$.

If $\gamma$ is a finite signed measure on $\R^N$ and $\f$ is a class of functions on $\R^N$, we identify $\gamma$ with the linear functional $g\mapsto \gamma(g):=\int g d\gamma$ on $\f$ and denote its $\ell^{\infty}(\f)$-norm by $\norm{\gamma}_{\f}:=\sup\{\gamma(g) : g\in \f\}$.

\section{Proof of Theorem \ref{thm:fast_rate}}\label{sec:proof}

Unless otherwise stated, the parameters $1<p<\infty$ and $\sigma>0$ are fixed throughout the remainder of this note. Furthermore we always assume that $T\ge 2$ and $d\ge 1$.

\subsection{Outline of the proof}\label{sec:outline proof}

The proof of Theorem~\ref{thm:fast_rate} is divided into three parts: Section~\ref{sec:lipkernel}, Section~\ref{sec:dpp} and Section~\ref{sec:emp}. Before we proceed with a rigorous proof, let us sketch the main steps. For simplicity of exposition, we assume that $T=2$ and that $\mu$ is compactly supported (only) in this section.

\textbf{Step 1:} We show in Proposition~\ref{prop:lipkernel} in Section~\ref{sec:lipkernel},  that smoothed measures have locally Lipschitz kernels. In particular, when $T=2$ and $\mu$ is compactly supported, there exists a constant $C>0$ such that
\begin{align}
    \w_p((\mu^{\sigma})_{x_1}, (\mu^{\sigma})_{y_1})
    \le C\abs{x_1-y_1} \text{ for all } x_1, y_1\in \R^d.\label{eq:sketch lip}
\end{align}
Proposition~\ref{prop:lipkernel} follows from combining Lemma~\ref{lem:dualnorm} and Lemma~\ref{lem:dualnormcontrol}: Lemma~\ref{lem:dualnorm} shows that if a measure $\gamma$ is absolutely continuous with respect to the Lebesgue measure, then
\begin{align}
    \w_p((\mu^{\sigma})_{x_1}, \gamma)
    \le C \norm{(\mu^{\sigma})_{x_1}-\gamma}_{\f},\label{eq:sketch dualnorm}
\end{align}
where $\f$ is a class of test functions specified in Lemma \ref{lem:dualnorm}. Setting $\gamma=(\mu^{\sigma})_{y_1}$, this implies that
\begin{align}
    \w_p((\mu^{\sigma})_{x_1}, (\mu^{\sigma})_{y_1})
    \le C \norm{(\mu^{\sigma})_{x_1}-(\mu^{\sigma})_{y_1}}_{\f}.
\end{align}
In Lemma~\ref{lem:dualnormcontrol}, we apply a Taylor expansion to show that
\begin{align}
    \norm{(\mu^{\sigma})_{x_1}-(\mu^{\sigma})_{y_1}}_{\f}
    \le C\abs{x_{1}-y_{1}}.
\end{align}

\textbf{Step 2:} In Section~\ref{sec:dpp}, we first recall the dynamic programming principle (DPP) for $\a\w_p$ in Proposition~\ref{prop:dpp}. When $T=2$, it can be stated as
\begin{align}
    \a\w^{(\sigma)}_p(\mu, \mhat{\mu}_n)^p
    =\inf_{\gamma\in \cpl((\mu^{\sigma})_1, (\mhat{\mu}^{\sigma}_n)_1)}\int \abs{x_1-y_1}^p +\w_p((\mu^{\sigma})_{x_1}, (\mhat{\mu}^{\sigma}_n)_{y_1})^p \gamma(dx_1, dy_1).\label{eq:sketch dpp}
\end{align}
Using \eqref{eq:sketch lip} together with the triangle inequality, we find
\begin{align}
    \w_p((\mu^{\sigma})_{x_1}, (\mhat{\mu}^{\sigma}_n)_{y_1})^p
    &\le C \w_p((\mu^{\sigma})_{x_1}, (\mu^{\sigma})_{y_1})^p
    +C\w_p((\mu^{\sigma})_{y_1}, (\mhat{\mu}^{\sigma}_n)_{y_1})^p\\
    &\le C \abs{x_1-y_1}^p
    +C \w_p((\mu^{\sigma})_{y_1}, (\mhat{\mu}^{\sigma}_n)_{y_1})^p.\label{eq:sketch W_p compact}
\end{align}
Plugging~\eqref{eq:sketch W_p compact} into \eqref{eq:sketch dpp} and choosing $\gamma$ as an optimal coupling for $\w_p^{(\sigma)}(\mu_1, (\mhat{\mu}_n)_1)$,
\begin{align}
    \a\w^{(\sigma)}_p(\mu, \mhat{\mu}_n)
    \le C\w^{(\sigma)}_p(\mu_1, (\mhat{\mu}_n)_1)
    +C\left(\int \w_p((\mu^{\sigma})_{y_1}, (\mhat{\mu}^{\sigma}_n)_{y_1})^p \mhat{\mu}^{\sigma}_n(dy)\right)^{1/p}.\label{eq:sketch dppbound}
\end{align}
Section~\ref{sec:dpp} is devoted to establishing Lemma~\ref{lem:dppbound}, which is an analogue of \eqref{eq:sketch dppbound} for subgaussian measures  $\mu$ that are not necessarily compactly supported. Let us emphasize here that if $\mu$ is not compactly supported, the kernel of $\mu^{\sigma}$ is \textit{locally} Lipschitz. Consequently, the estimate~\eqref{eq:sketch lip} may fail and estimates~\eqref{eq:sketch W_p compact} and~\eqref{eq:sketch dppbound} have to be modified accordingly. 

\textbf{Step 3:} In Section~\ref{sec:emp}, we complete the proof of Theorem~\ref{thm:fast_rate}. The main idea is to partition $\R^d$ into disjoint sets $(E_k)_{k=0}^{\infty}$ and apply empirical process theory (see Lemma~\ref{lem:emp}). By taking expectations in~\eqref{eq:sketch dppbound} and using \eqref{eq:sketch dualnorm}, we can show that
\begin{align}
    \E[\a\w^{(\sigma)}_p(\mu, \mhat{\mu}_n)]
    &\le C\E[\w^{(\sigma)}_p(\mu_1, (\mhat{\mu}_n)_1)]
    +C\E\left[\left(\int \w_p((\mu^{\sigma})_{y_1}, (\mhat{\mu}^{\sigma}_n)_{y_1})^p \mhat{\mu}^{\sigma}_n(dy)\right)^{1/p}\right]\\
    &\le C\E[\w^{(\sigma)}_p(\mu_1, (\mhat{\mu}_n)_1)]
    +C\E\left[\left(\int \norm{(\mu^{\sigma})_{y_1}- (\mhat{\mu}^{\sigma}_n)_{y_1}}^p_{\f} \mhat{\mu}^{\sigma}_n(dy)\right)^{1/p}\right]\\
    &\le C\E[\w^{(\sigma)}_p(\mu_1, (\mhat{\mu}_n)_1)]
    +C\sum_{k=0}^{\infty}\E\left[\left(\int_{E_k} \norm{(\mu^{\sigma})_{y_1}- (\mhat{\mu}^{\sigma}_n)_{y_1}}^p_{\f} \mhat{\mu}^{\sigma}_n(dy)\right)^{1/p}\right].
\end{align}
In Lemma~\ref{lem:dpp norm bound}, we further refine this estimate and show that
\begin{align}
    \E[\a\w^{(\sigma)}_p(\mu, \mhat{\mu}_n)]
    \le Cn^{-1/2}+C\sum_{k=0}^{\infty}f(k)\E[\norm{\mu-\mhat{\mu}_n}_{\h_k}]\label{eq:sketch dpp norm bound}
\end{align}
for some class of functions $\h_k$ and a function $f$. In Lemma~\ref{lem:h bound}, we apply empirical process theory to \eqref{eq:sketch dpp norm bound} to obtain that
\begin{align}
    \E[\norm{\mu-\mhat{\mu}_n}_{\h_k}]\le Cn^{-1/2}\sum_{\ell=0}^{\infty}g(k, \ell)
\end{align}
for some function $g$. As a final step, we show how the moment condition of $\mu$ implies that the series
\begin{align}
    \sum_{k=0}^{\infty}f(k)\sum_{\ell=0}^{\infty}g(k, \ell)
\end{align}
converges.

To improve readability, technical lemmas involving computations or estimates of Gaussian kernels are deferred to Appendix~\ref{appendix:technical lemmas}.

\subsection{Step 1: Regularity of kernels of smoothed measures}\label{sec:lipkernel}

This section is mainly devoted to the proof of Proposition \ref{prop:lipkernel}. As mentioned in Section \ref{sec:intro}, a subgaussian measure convolved with Gaussian noise has locally Lipschitz kernels. Moreover, if the measure is compactly supported, then its kernels are Lipschitz.

\begin{prop}[Kernels of a smoothed measure]\label{prop:lipkernel}
    Let $0<\beta<1/p'$. Suppose that $\mu\in \sp((\R^d)^T)$ satisfies $\e_{q/(2\sigma^2)}(\mu)<\infty$ for $q>2(p-1)/\beta$. Then there exists a constant $C>0$ that depends only on $d, p, \sigma, q, \beta, \e_{q/(2\sigma^2)}(\mu)$ such that for all $x,y\in (\R^d)^T$ and $t\in \{1,2,\ldots, T-1\}$,
    \begin{align}
        \w_p((\mu^{\sigma})_{x_{1:t}}, (\mu^{\sigma})_{y_{1:t}})
        \le C e^{\frac{\beta}{\sigma^2}\left(\abs{x_{1:t}-\mean(\mu_{1:t})}\vee \abs{y_{1:t}-\mean(\mu_{1:t})}\right)^2}
        \abs{x_{1:t}-y_{1:t}}.\label{eq: exp lipkernel}
    \end{align}
    In particular, if $\mu$ is compactly supported,
    \begin{align}
        \w_p((\mu^{\sigma})_{x_{1:t}}, (\mu^{\sigma})_{y_{1:t}})
        \le C \abs{x_{1:t}-y_{1:t}}
    \end{align}
    for some constant $C>0$ that depends only on $d, p, \sigma, \supp(\mu)$.
\end{prop}

The proof of Proposition~\ref{prop:lipkernel} is deferred to the end of this section. The proof is based on two lemmas: Lemma~\ref{lem:dualnorm} and Lemma~\ref{lem:dualnormcontrol}. We first recall \cite[Proposition $2.1$]{goldfeld2024limit} which shows that the Wasserstein distance $\w_p$ can be bounded above by the dual Sobolev norm.

\begin{prop}[Proposition $2.1$ in \cite{goldfeld2024limit}]\label{prop:dualsobolev}
    Let $1\le p<\infty$ and suppose that $\mu_0, \mu_1\in \sp(\R^d)$ with $\mu_0, \mu_1\ll \rho$ for some reference measure $\rho\in \sp(\R^d)$. Denote their respective densities by $f_i=d\mu_i/d\rho$, $i=0,1$. If $f_0$ or $f_1$ is bounded from below by some $c>0$, then
    \begin{align}\label{eq:dualsobolev}
        \w_p(\mu_0, \mu_1)
        \le p c^{-1/p'}
        \sup\left\{(\mu_1-\mu_0)(\psi) : \psi\in C^{\infty}_c(\R^d), \norm{\nab \psi}_{L^{p'}(\rho; \R^d)} \le 1\right\}.
    \end{align}
\end{prop}

\begin{rmk}
    When $p=1$, a density argument applied to the Kantorovich--Rubinstein duality for $\w_1(\mu_0, \mu_1)$ shows that equality holds in \eqref{eq:dualsobolev}. Note that the supremum in \eqref{eq:dualsobolev} is the operator norm of $(\mu_1-\mu_0)$ in the dual of the homogeneous Sobolev space $\dot{H}^{1,p'}$. See \cite{nietert2021smooth, goldfeld2024limit} for details.
\end{rmk}

We will apply Proposition \ref{prop:dualsobolev} to kernels of smoothed measures: in particular, Lemma~\ref{lem:dualnorm} is obtained by applying Proposition~\ref{prop:dualsobolev} to $\mu_0=(\mu^{\sigma})_{x_{1:t}}$. To proceed, we define a function class $\f^{\sigma, p}$ for $1<p<\infty$, which is essentially the function class appearing in \eqref{eq:dualsobolev} adapted to our setting:
\begin{align}\label{eq:fc_f}
    \f^{\sigma, p}
    :=\{\psi-\n_{\sigma\eta}(\psi) : \psi\in C^{\infty}_c(\R^d), \norm{\nab \psi}_{L^{p'}(\n_{\sigma\eta}; \R^d)}\le 1\} \text{ where } \eta:=\sqrt{1/(2p)'} = \sqrt{1-\frac{1}{2p}}.
\end{align}

The reason for the  choice $\rho=\n_{\sigma\eta}$ will become apparent below.

\begin{lem}\label{lem:dualnorm}
    Let $0<\beta<1/p'$. Suppose that $\mu\in \sp((\R^d)^T)$ satisfies $\e_{q_0/(2\sigma^2)}(\mu)<\infty$, where $q_0=2(p-1)(1/\beta-p')>0$. If $\gamma\in \sp(\R^d)$ is absolutely continuous with respect to the Lebesgue measure, then for all $x\in (\R^d)^T$ and $t\in \{1,2,\ldots, T-1\}$,
    \begin{align}
        \w_p((\mu^{\sigma})_{x_{1:t}}, \gamma)
        \le C
        e^{\frac{\beta}{2\sigma^2}\abs{x_{1:t}-\mean\left(\mu_{1:t}\right)}^2}
        \norm{(\mu^{\sigma})_{x_{1:t}}-\gamma}_{\f^{\sigma, p}},
    \end{align}
    where $C=p ((2p)')^{d/(2p')}
        (\mathcal{E}_{q_0/(2\sigma^2)}(\mu_{t+1}))^{\beta/(1-\beta p')}
        e^{\frac{\beta }{2\sigma^2}\var\left(\mu_{1:t}\right)}$.
\end{lem}

\begin{proof}
Recall $\eta =\sqrt{1/(2p)'}$ as defined in \eqref{eq:fc_f}. The proof follows from applying Proposition \ref{prop:dualsobolev} to the reference measure $\rho:=\n_{\sigma \eta}$, once we have shown the following lower bound for the density:
\begin{align}\label{eq:lower_bound}
    \frac{d(\mu^{\sigma})_{x_{1:t}}}{d\n_{\sigma\eta}}(x_{t+1})
    &=\frac{\varphi_{\sigma}\ast \mu(x_{1:t+1})}{\varphi_{\sigma}\ast \mu(x_{1:t})\varphi_{\sigma\eta}(x_{t+1})}\\
    &\ge \eta^{d} e^{-\frac{\beta p'}{2\sigma^2}\var(\mu_{1:t})}(\e_{q_0/(2\sigma^2)}(\mu_{t+1}))^{-\beta p'/(1-\beta p')}e^{-\frac{\beta p'}{2\sigma^2}\abs{x_{1:t}-\mean(\mu_{1:t})}^2}.
\end{align}
The well-known inequality $\vert x_{t+1}-y_{t+1} \vert^2\le \vert x_{t+1} \vert^2/\eta^2+\vert y_{t+1} \vert^2/(1-\eta^2)$ implies
\begin{align}
    \varphi_{\sigma}(x_{t+1}-y_{t+1})
    =(2\pi\sigma^2)^{-d/2}e^{-\frac{\abs{x_{t+1}-y_{t+1}}^2}{2\sigma^2}}
    \ge \eta^d\varphi_{\sigma \eta}(x_{t+1})e^{-\frac{\abs{y_{t+1}}^2}{2\sigma^2(1-\eta^2)}}.
\end{align}
We thus conclude that
\begin{align*}
    \varphi_{\sigma}\ast \mu(x_{1:t+1})
    &=\int \varphi_{\sigma}(x_{1:t}-y_{1:t})\varphi_{\sigma}(x_{t+1}-y_{t+1})\mu(dy)\\
    &\ge \eta^d\varphi_{\sigma \eta}(x_{t+1}) \int \varphi_{\sigma}(x_{1:t}-y_{1:t})e^{-\frac{\abs{y_{t+1}}^2}{2\sigma^2(1-\eta^2)}}\mu(dy).
\end{align*}
We now apply the reverse H\"older inequality
\begin{align}
    \int fg d\mu \ge \left(\int f^{1/r}d\mu\right)^{r}\left(\int g^{-1/(r-1)}d\mu\right)^{-(r-1)}
\end{align}
with $r= \frac{1}{1-\beta p'}>1$, $f(y_{1:t})=\varphi_{\sigma}(x_{1:t}-y_{1:t})$ and $g(y_{t+1})=e^{-\frac{\abs{y_{t+1}}^2}{2\sigma^2(1-\eta^2)}}$. Noting that
\begin{align}
    \frac{1}{r}=1-\beta p', \quad
    \frac{1}{1-\eta^2}\frac{1}{r-1}=2p \frac{1-\beta p'}{\beta p'}=q_0,
\end{align}
this gives
\begin{align}
&\eta^d\varphi_{\sigma \eta}(x_{t+1}) \int \varphi_{\sigma}(x_{1:t}-y_{1:t})e^{-\frac{\abs{y_{t+1}}^2}{2\sigma^2(1-\eta^2)}}\mu(dy)\\
&\ge \eta^d \varphi_{\sigma \eta}(x_{t+1}) (\varphi^{1-\beta p'}_{\sigma}\ast \mu(x_{1:t}))^{\frac{1}{1-\beta p'}}(\e_{q_0/(2\sigma^2)}(\mu_{t+1}))^{-\beta p'/(1-\beta p')}.\label{eq:lbratio1}
\end{align}
Choosing $a_1=1-\beta p'<1=a_2$ in Lemma \ref{lem:lbofdensity}\ref{lem:lbofdensity_a} (Appendix~\ref{appendix:technical lemmas}) and noting that $(1-a_1/a_2)=\beta p'$, we obtain
\begin{align}
    (\varphi^{1-\beta p'}_{\sigma}\ast \mu(x_{1:t}))^{\frac{1}{1-\beta p'}}
    &=(\varphi^{a_1}_{\sigma}\ast \mu(x_{1:t}))^{1/a_1}\\
    &\ge (\varphi^{a_2}_{\sigma}\ast \mu(x_{1:t}))^{1/a_2}
    e^{-\frac{(1-a_1/a_2)}{2\sigma^2}\abs{x_{1:t}-\mean(\mu_{1:t})}^2}
    e^{-\frac{(1-a_1/a_2)}{2\sigma^2}\var(\mu_{1:t})}\\
    &= \varphi_{\sigma}\ast \mu(x_{1:t})
    e^{-\frac{\beta p'}{2\sigma^2}\abs{x_{1:t}-\mean(\mu_{1:t})}^2}
    e^{-\frac{\beta p'}{2\sigma^2}\var(\mu_{1:t})}.\label{eq:lbratio2}
\end{align}
Combining \eqref{eq:lbratio1} and \eqref{eq:lbratio2} yields the desired result.
\end{proof}

In Lemma~\ref{lem:dualnorm}, we take $\gamma=(\mu^{\sigma})_{y_{1:t}}$ and observe that
\begin{align}
    \w_p((\mu^{\sigma})_{x_{1:t}}, (\mu^{\sigma})_{y_{1:t}})
    \le C
    e^{\frac{\beta}{2\sigma^2}\abs{x_{1:t}-\mean\left(\mu_{1:t}\right)}^2}
    \norm{(\mu^{\sigma})_{x_{1:t}}-(\mu^{\sigma})_{y_{1:t}}}_{\f^{\sigma, p}}.\label{eq:step after dualnorm}
\end{align}
Consequently, Proposition~\ref{prop:lipkernel} follows from Lemma~\ref{lem:dualnormcontrol}, which shows that the supremum norm over $\f^{\sigma, p}$ in \eqref{eq:step after dualnorm} is bounded above by $\abs{x_{1:t}-y_{1:t}}$ up to a constant factor. The proof of Lemma~\ref{lem:dualnormcontrol} uses a Taylor expansion in conjunction with Lemmas~\ref{lem:lbofdensity} and~\ref{lem:gaussiancomputation} in Appendix~\ref{appendix:technical lemmas}.

\begin{lem}\label{lem:dualnormcontrol}
    Let $0<\beta<1$. Suppose that $\mu\in \sp((\R^d)^T)$ satisfies $\e_{q/(2\sigma^2)}(\mu)<\infty$, where $q>2(p-1)/\beta$. Then for all $x, y\in (\R^d)^T$ and $t\in \{1,2,\ldots, T-1\}$,
    \begin{align}
        \norm{(\mu^{\sigma})_{x_{1:t}}-(\mu^{\sigma})_{y_{1:t}}}_{\f^{\sigma, p}}
        \le C e^{\frac{\beta}{2\sigma^2}\left(\abs{x_{1:t}-\mean(\mu_{1:t})}\vee \abs{y_{1:t}-\mean(\mu_{1:t})}\right)^2}\abs{x_{1:t}-y_{1:t}},
    \end{align}
    where
    \begin{align}
        C=\sigma^{-2}
        D_{p, d, \sigma}(2^{1/p}/(2p)')^{d/2}
        e^{\frac{\beta }{2\sigma^2}\var(\mu_{1:t})}
        \left(\norm{h}_{L^{1/\beta}(\mu)}+(M_r(\mu_{1:t}))^{1/r}(\e_{q/(2\sigma^2)}(\mu_{1:t}))^{2(p-1)/q}\right)
    \end{align}
    with $D_{p, d, \sigma}$ defined in Lemma~\ref{lem:gaussiancomputation}, $h(w)=\abs{w_{1:t}}e^{\frac{(p-1)\abs{w_{t+1}}^2}{\sigma^2}}$ and $r=\frac{q}{\beta q-2(p-1)}$.
\end{lem}
\begin{proof}
For $f\in \f^{\sigma, p}$, note that
\begin{align}
    ((\mu^{\sigma})_{x_{1:t}}-(\mu^{\sigma})_{y_{1:t}})(f)
    =\int f(z_{t+1})\left(\frac{\varphi_{\sigma}\ast\mu(x_{1:t}, z_{t+1})}{\varphi_{\sigma}\ast\mu(x_{1:t})}-\frac{\varphi_{\sigma}\ast\mu(y_{1:t}, z_{t+1})}{\varphi_{\sigma}\ast\mu(y_{1:t})}\right) dz_{t+1}.
\end{align}
Let $x_{1:t}(s):=sx_{1:t}+(1-s)y_{1:t}$ for $s\in [0,1]$ and set $ g(s):=\frac{\varphi_{\sigma}\ast\mu(x_{1:t}(s), z_{t+1})}{\varphi_{\sigma}\ast\mu(x_{1:t}(s))}$. From Fubini's theorem and $g(1)-g(0)=\int_0^1 g'(s)ds$, we have
\begin{align}
    ((\mu^{\sigma})_{x_{1:t}}-(\mu^{\sigma})_{y_{1:t}})(f)
    \le \int_0^1\left(\int \abs{f(z_{t+1})}\abs{g'(s)}dz_{t+1}\right)ds. \label{eq:lipfunctional}
\end{align}
We first claim that
\begin{align}
    \abs{g'(s)}
    \le \frac{\abs{x_{1:t}-y_{1:t}}}{\sigma^2} \left(\int \varphi_{\sigma}(z_{t+1}-w_{t+1})\abs{w_{1:t}} \kappa(dw)
    +g(s)\int \abs{w_{1:t}}\kappa(dw_{1:t})\right)\label{eq:bound g'(s)}
\end{align}
where $\kappa\in \sp((\R^d)^T)$ is defined via 
\begin{align}
    \kappa(A):=\frac{1}{\varphi_{\sigma}\ast \mu(x_{1:t}(s))}\int_{A} \varphi_{\sigma}(x_{1:t}(s)-w_{1:t})\mu(dw_{1:T}).
\end{align}
To see this, we compute
\begin{align}
    &\frac{d}{ds}\log(\varphi_{\sigma}\ast\mu(x_{1:t}(s), z_{t+1}))\\
    &=-\frac{1}{\varphi_{\sigma}\ast \mu(x_{1:t}(s), z_{t+1})}\int \varphi_{\sigma}(x_{1:t}(s)-w_{1:t}, z_{t+1}-w_{t+1})\Big(\frac{x_{1:t}(s)-w_{1:t}}{\sigma^2}\Big)\cdot \left(x_{1:t}-y_{1:t}\right)\mu(dw_{1:t+1})\\
    &=-\frac{x_{1:t}(s)\cdot(x_{1:t}-y_{1:t})}{\sigma^2}
    +\frac{1}{\sigma^2 g(s)}\int \varphi_{\sigma}(z_{t+1}-w_{t+1}) w_{1:t}\cdot (x_{1:t}-y_{1:t}) \kappa(dw).
\end{align}
Similarly,
\begin{align}
    \frac{d}{ds}\log(\varphi_{\sigma}\ast\mu(x_{1:t}(s)))
    =-\frac{x_{1:t}(s)\cdot(x_{1:t}-y_{1:t})}{\sigma^2}
    +\frac{1}{\sigma^2}\int w_{1:t}\cdot(x_{1:t}-y_{1:t}) \kappa(dw).
\end{align}
Canceling terms
\begin{align}
    \frac{g'(s)}{g(s)}
    =\frac{d}{ds}\log(g(s))
    =\frac{1}{\sigma^2 g(s)}\int \varphi_{\sigma}(z_{t+1}-w_{t+1}) w_{1:t}\cdot (x_{1:t}-y_{1:t}) \kappa(dw)
    -\frac{1}{\sigma^2}\int w_{1:t}\cdot(x_{1:t}-y_{1:t}) \kappa(dw).
\end{align}
This establishes \eqref{eq:bound g'(s)}.

Next, we apply \eqref{eq:bound g'(s)} along with Fubini's theorem to obtain
\begin{align}
    &\int \abs{f(z_{t+1})} \abs{g'(s)} dz_{t+1}\\
    &\le \frac{\abs{x_{1:t}-y_{1:t}}}{\sigma^2} \left(\int \abs{f}\ast \varphi_{\sigma}(w_{t+1})\abs{w_{1:t}} \kappa(dw)
    +\int \abs{f(z_{t+1})}g(s)dz_{t+1} \int \abs{w_{1:t}}\kappa(dw)\right)\\
    &=: \frac{\abs{x_{1:t}-y_{1:t}}}{\sigma^2}(\text{I}+\text{II}).\label{eq:f g' bound}
\end{align}
It remains to bound $\text{I}+\text{II}$. For this, applying Lemma \ref{lem:gaussiancomputation}\ref{lem:gaussiancomputation_a} and Lemma \ref{lem:lbofdensity}\ref{lem:lbofdensity_b} (Appendix~\ref{appendix:technical lemmas}) to $x_{1:t}(s)$ instead of $x_{1:t}$,
\begin{align}
    \text{I}
    &\le D_{p, d, \sigma}(2^{1/p}/(2p)')^{d/2}
    \int \abs{w_{1:t}}e^{\frac{(p-1)\abs{w_{t+1}}^2}{\sigma^2}} \kappa(dw)\\
    &\le D_{p, d, \sigma}(2^{1/p}/(2p)')^{d/2}
    \norm{h}_{L^{1/\beta}(\mu)}
    e^{\frac{\beta}{2\sigma^2}\var(\mu_{1:t})}
    e^{\frac{\beta}{2\sigma^2}\abs{x_{1:t}(s)-\mean(\mu_{1:t})}^2},\label{eq:I control}
\end{align}
where $h(w):=\abs{w_{1:t}}e^{\frac{(p-1)\abs{w_{t+1}}^2}{\sigma^2}}$. Similarly, we can establish
\begin{align}
    &\text{II}
    \le D_{p, d, \sigma}(2^{1/p}/(2p)')^{d/2}
    \left((M_r(\mu_{1:t}))^{1/r}(\e_{q/(2\sigma^2)}(\mu_{t+1}))^{2(p-1)/q}\right)\\
    &\hspace{5cm}\cdot e^{\frac{\beta}{2\sigma^2}\var(\mu_{1:t})}
    e^{\frac{\beta}{2\sigma^2}\abs{x_{1:t}(s)-\mean(\mu_{1:t})}^2},\label{eq:II control}
\end{align}
where $r:=\frac{q}{\beta q -2(p-1)}>1$. Indeed, Lemma \ref{lem:gaussiancomputation}\ref{lem:gaussiancomputation_b} (Appendix~\ref{appendix:technical lemmas}) shows
\begin{align}
    \text{II}
    &=\frac{1}{\varphi_{\sigma}\ast\mu(x_{1:t}(s))}\int \abs{f(z_{t+1})}\varphi_{\sigma}\ast \mu(x_{1:t}(s), z_{t+1})dz_{t+1}
    \int \abs{w_{1:t}}\kappa(dw)\\
    &\le D_{p, d, \sigma}(2^{1/p}/(2p)')^{d/2}
    \int e^{\frac{(p-1)\abs{w_{t+1}}^2}{\sigma^2}} \kappa(dw)
    \int \abs{w_{1:t}}\kappa(dw).
\end{align}
To see \eqref{eq:II control}, we apply Lemma \ref{lem:lbofdensity}\ref{lem:lbofdensity_b} (Appendix~\ref{appendix:technical lemmas}) to $a=\beta-2(p-1)/q=1/r\in (0,1)$ for $x_{1:t}(s)$, which yields
\begin{align}
    \int \abs{w_{1:t}}\kappa(dw)
    \le (M_{r}(\mu_{1:t}))^{1/r}
    e^{\frac{\beta-2(p-1)/q}{2\sigma^2}\var(\mu_{1:t})}
    e^{\frac{\beta-2(p-1)/q}{2\sigma^2}\abs{x_{1:t}(s)-\mean(\mu_{1:t})}^2}.
\end{align}
By choosing $a=2(p-1)/q\in (0,1)$ in Lemma \ref{lem:lbofdensity}\ref{lem:lbofdensity_b} (Appendix~\ref{appendix:technical lemmas}),
\begin{align}
    \int e^{\frac{(p-1)\abs{w_{t+1}}^2}{\sigma^2}} \kappa(dw)
    \le (\e_{q/(2\sigma^2)}(\mu_{t+1}))^{2(p-1)/q}
    e^{\frac{2(p-1)/q}{2\sigma^2}\var(\mu_{1:t})}
    e^{\frac{2(p-1)/q}{2\sigma^2}\abs{x_{1:t}(s)-\mean(\mu_{1:t})}^2}.
\end{align}
This shows \eqref{eq:II control}. By plugging \eqref{eq:I control} and \eqref{eq:II control} into \eqref{eq:lipfunctional} and using $\vert x_{1:t}(s)-\mean(\mu_{1:t})\vert\le \vert x_{1:t}-\mean(\mu_{1:t})\vert\vee \vert y_{1:t}-\mean(\mu_{1:t})\vert$, this concludes the proof of the lemma.
\end{proof}

\begin{proof}[Proof of Proposition \ref{prop:lipkernel}]
It is easy to check that the moment assumptions of Lemma \ref{lem:dualnorm} and Lemma \ref{lem:dualnormcontrol} are satisfied. We denote the constant appearing in Lemma \ref{lem:dualnorm} by $C_1$,  and the constant appearing in Lemma \ref{lem:dualnormcontrol} by $C_2$. Then from Lemma \ref{lem:dualnorm} and Lemma \ref{lem:dualnormcontrol},
\begin{align}\label{eq:lipkernel_dualnorm}
    \w_p((\mu^{\sigma})_{x_{1:t}}, (\mu^{\sigma})_{y_{1:t}})
    &\le C_1 e^{\frac{\beta}{2\sigma^2}\abs{x_{1:t}-\mean(\mu_{1:t})}^2}\norm{(\mu^{\sigma})_{x_{1:t}}-(\mu^{\sigma})_{y_{1:t}}}_{\f^{\sigma, p}}\\
    &\le C_1C_2 e^{\frac{\beta}{2\sigma^2}\abs{x_{1:t}-\mean(\mu_{1:t})}^2}e^{\frac{\beta}{2\sigma^2}\left(\abs{x_{1:t}-\mean(\mu_{1:t})}\vee \abs{y_{1:t}-\mean(\mu_{1:t})}\right)^2}\abs{x_{1:t}-y_{1:t}}.
\end{align}
Since the moments of $\mu$ appearing in $C_1$ and $C_2$ can be all bounded from above by $\e_{q/(2\sigma^2)}(\mu)$, this proves the general case. When $\mu$ is compactly supported, we can get the desired result by sending $\beta\to 0$. Indeed, by translating the support of $\mu$ if necessary, we may assume that $0\in \supp(\mu)$. Next we note that the constant $C_1$ in Lemma \ref{lem:dualnorm} has a finite limit as $\beta\to 0$. Indeed, as $\beta\to 0$,
\begin{align}
    C_1
    =p((2p)')^{d/(2p')}
    \left(\mathcal{E}_{\frac{p-1}{\sigma^2}(\frac{1}{\beta}-p')}(\mu_{t+1})\right)^{\beta/(1-\beta p')}
    e^{\frac{\beta }{2\sigma^2}\var(\mu_{1:t})}
    \longrightarrow
    p((2p)')^{d/(2p')}\norm{g}^{(p-1)/\sigma^2}_{L^{\infty}(\mu_{t+1})},
\end{align}
where $g(x_{t+1}):=e^{\abs{x_{t+1}}^2}$. In particular, we bound
\begin{align}
    \lim_{\beta\to 0}C_1\le p((2p)')^{d/(2p')}e^{\frac{(p-1)\diam(\supp(\mu_{t+1}))}{\sigma^2}}.
\end{align}
Similarly, the constant $C_2$ has a finite limit as $\beta\to 0$ and this limit can be controlled using $p,d,\sigma, \supp(\mu)$. This proves Proposition \ref{prop:lipkernel}.
\end{proof}

\subsection{Step 2: DPP-based estimates for \texorpdfstring{$\a\w_p$}{TEXT}}\label{sec:dpp}

We first recall the dynamic programming principle (DPP) for the adapted Wasserstein distance \cite[Proposition~5.1]{backhoff2017causal}, which we present next.

\begin{prop}[Proposition $5.1$ in \cite{backhoff2017causal}]\label{prop:dpp}
Given $\mu, \nu\in \sp_p((\R^d)^T)$, let us define $V_t:(\R^d)^t\times (\R^d)^t\to \R$, $t\in \{1,2,\ldots, T\}$ and $V_0\in \R$ via the following recursive formula: $V_T =0 $ and for $t\in \{0,1,\ldots, T-1\}$,
\begin{align}
    V_{t}(x_{1:t}, y_{1:t})
    =\inf_{\gamma_{x_{1:t}, y_{1:t}}\in \cpl(\mu_{x_{1:t}}, \nu_{y_{1:t}})} \int \abs{x_{t+1}-y_{t+1}}^p + V_{t+1}(x_{1:t+1}, y_{1:t+1})\gamma_{x_{1:t}, y_{1:t}}(dx_{t+1}, dy_{t+1}),
\end{align}
where $\mu_{x_{1:0}}:=\mu_1$ and $\nu_{y_{1:0}}:=\nu_1$. Then $V_0=\a\w_p(\mu, \nu)^p$.
\end{prop}

In this section, we aim to prove Lemma~\ref{lem:dppbound}, which states that $\a\w^{(\sigma)}_p(\mu, \mhat{\mu}_n)$ can be bounded by a sum of $\w_{2p}$-distances between the kernels of $\mu^{\sigma}$ and $\mhat{\mu}^{\sigma}_n$ under suitable moment assumptions on $\mu$. To this end, we apply Proposition~\ref{prop:lipkernel} and then incorporate it into the DPP given in Proposition~\ref{prop:dpp}. For the statement of Lemma~\ref{lem:dppbound}, we recall the projection map $P^1((x_1, \ldots, x_T))=x_1$ and $\mu_1=P^{1}_{\#}\mu$ and $(\mhat{\mu}_n)_1=P^1_{\#}\mhat{\mu}_n$.

\begin{lem}[DPP]\label{lem:dppbound}
Let $\beta$ satisfy $\eqref{assume:parameters}$. Suppose that $\mu\in \sp((\R^d)^T)$ satisfies $\mean(\mu)=0$ and $\e_{q/(2\sigma^2)}(\mu)<\infty$ where $q>q^{*}(p, T,\beta)$. Then there exists a constant $C>0$ that depends only on $d$,$T$,$p$,$\sigma$,$q$,$\beta$, $\e_{q/(2\sigma^2)}(\mu)$ such that
\begin{equation}\label{eq:lem dppbound}
\begin{aligned}
    \a\w^{(\sigma)}_p(\mu, \mhat{\mu}_n)
    \le &C \Big(\e_{\frac{2p\beta(T-1)}{\sigma^2(1-4p\beta(T-1))}}(\mhat{\mu}_n)\Big)^{1/(2p)}\\
    &\cdot\left(\w^{(\sigma)}_{2p}(\mu_1, (\mhat{\mu}_n)_1)
    +\sum_{t=1}^{T-1}\left(
    \int \w_{2p}((\mu^{\sigma})_{y_{1:t}}, (\mhat{\mu}^{\sigma}_n)_{y_{1:t}})^{2p} \mhat{\mu}^{\sigma}_n(dy)
    \right)^{1/(2p)}\right) \,\,\text{a.s.}
\end{aligned}
\end{equation}
\end{lem}

Before presenting the proof of Lemma~\ref{lem:dppbound}, we first illustrate how the DPP together with Proposition~\ref{prop:lipkernel} is used to establish an upper bound for $\a\w^{(\sigma)}_p(\mu, \mhat{\mu}_n)$. When $\mu$ is compactly supported, this argument was already outlined in Step 2 in Section~\ref{sec:outline proof}. Here, we extend the argument to the subgaussian case. This naturally motivates the parameter condition \eqref{assume:parameters} and the moment condition $q^{*}(p, T, \beta)$ appearing in Lemma~\ref{lem:dppbound}. The precise definitions of \eqref{assume:parameters} and $q^{*}(p, T, \beta)$ will be given later in this section.

We now consider the case where $T=2$ and $\mu$ is subgaussian. From Proposition~\ref{prop:dpp},
\begin{align}
    \a\w^{(\sigma)}_p(\mu, \mhat{\mu}_n)^p
    =\inf_{\gamma\in \cpl((\mu^{\sigma})_1, (\mhat{\mu}^{\sigma}_n)_1)}\int \abs{x_1-y_1}^p +\w_p((\mu^{\sigma})_{x_1}, (\mhat{\mu}^{\sigma}_n)_{y_1})^p \gamma(dx_1, dy_1).\label{eq:ex T=2 dpp}
\end{align}
By the triangle inequality,
\begin{align}
    \w_p((\mu^{\sigma})_{x_1}, (\mhat{\mu}^{\sigma}_n)_{y_1})^p
    \le C \w_p((\mu^{\sigma})_{x_1}, (\mu^{\sigma})_{y_1})^p
    +C\w_p((\mu^{\sigma})_{y_1}, (\mhat{\mu}^{\sigma}_n)_{y_1})^p.
\end{align}
Recall from Proposition~\ref{prop:lipkernel} that the kernels of $\mu^{\sigma}$ are locally Lipschitz with additional exponential functions appearing in \eqref{eq: exp lipkernel}. This results in
\begin{align}
    \w_p((\mu^{\sigma})_{x_1}, (\mhat{\mu}^{\sigma}_n)_{y_1})^p
    \le   Ce^{\frac{p\beta}{\sigma^2}\left(\abs{x_{1}-\mean(\mu_{1})}\vee \abs{y_{1}-\mean(\mu_1 )}\right)^2}
    \abs{x_1-y_1}^p
    +C\w_p((\mu^{\sigma})_{y_1}, (\mhat{\mu}^{\sigma}_n)_{y_1})^p.
\end{align}
Plugging this back into \eqref{eq:ex T=2 dpp},
\begin{align}
\begin{split}
    \a\w^{(\sigma)}_p(\mu, \mhat{\mu}_n)^p
    &\le C\inf_{\gamma\in \cpl((\mu^{\sigma})_1, (\mhat{\mu}^{\sigma}_n)_1)}
    \int e^{\frac{p\beta}{\sigma^2}\left(\abs{x_{1}-\mean(\mu_{1})}\vee \abs{y_{1}-\mean(\mu_1 )}\right)^2}\abs{x_1-y_1}^p \gamma(dx_1, dy_1)\\
    &+C\int \w_p((\mu^{\sigma})_{y_1}, (\mhat{\mu}^{\sigma}_n)_{y_1})^p\mhat{\mu}^{\sigma}_n(dy).\label{eq:W_2p vs W_r}
\end{split}
\end{align}
Compared to the compactly supported case where $\beta$ can be chosen equal to zero, this bound is looser as $\mu$ is only assumed to be subgaussian. To control the infimum above, we choose $\gamma$ as an optimal coupling for $\w_{2p}((\mu^{\sigma})_1, (\mhat{\mu}^{\sigma}_n)_1)$ (\textit{not} for $\w_{p}$ as in Section~\ref{sec:outline proof}) and use the Cauchy--Schwarz inequality, which yields
\begin{align}
    \a\w^{(\sigma)}_p(\mu, \mhat{\mu}_n)^p
    &\le
    C\left(\int e^{\frac{2p\beta}{\sigma^2}\left(\abs{x_{1}-\mean(\mu_{1})}\vee \abs{y_{1}-\mean(\mu_1 )}\right)^2}\gamma(dx_1, dy_1)\right)^{1/2}
    \w_{2p}((\mu^{\sigma})_1, (\mhat{\mu}^{\sigma}_n)_1)^{p}\\
    &+C\int \w_p((\mu^{\sigma})_{y_1}, (\mhat{\mu}^{\sigma}_n)_{y_1})^p\mhat{\mu}^{\sigma}_n(dy).
\end{align}
Omitting details, this implies that under suitable moment assumptions on $\mu$,
\begin{equation}\label{eq:sketch dppbound subgaussian}
\begin{aligned}
    &\a\w^{(\sigma)}_p(\mu, \mhat{\mu}_n)^p\\
    &\le
    C\left(\e_{\frac{2p\beta}{\sigma^2(1-4p\beta)}}(\mhat{\mu}_n)\right)^{1/2}\left(\w^{(\sigma)}_{2p}(\mu_1, (\mhat{\mu}_n)_1)^p
    +C\left(\int \w_{2p}((\mu^{\sigma})_{y_{1}}, (\mhat{\mu}_n)_{y_{1}})^{2p}\mhat{\mu}^{\sigma}_n(dy)\right)^{1/2}\right).
\end{aligned}
\end{equation}
The main difference between the above estimate~\eqref{eq:sketch dppbound subgaussian} and the estimate \eqref{eq:sketch dppbound} in Section~\ref{sec:outline proof} is that the upper bound in \eqref{eq:sketch dppbound subgaussian} is \textit{looser} in the sense that it is stated in terms of $\w_{2p}$-distances rather than $\w_p$-distances.

In the proof of Lemma~\ref{lem:dppbound}, we will need to inductively apply Proposition~\ref{prop:lipkernel} to $\w_{2p}$ instead of $\w_p$. For simplicity of notation, we will henceforth fix a parameter $\beta$ satisfying \eqref{assume:parameters} below and define $q^{*}(p, T, \beta)$ as follows:
\begin{equation}\tag{P}\label{assume:parameters}
\begin{aligned}
    \quad 0<\beta<\min\left\{\frac{1}{4p(T-1)}, \frac{1}{8p}\right\},
\end{aligned}
\end{equation}
\begin{equation}\tag{Q}\label{assume:q star}
\begin{aligned}
    q^{*}(p, T, \beta)
    :=\max\Bigg\{\frac{2(2p-1)}{\beta},
    \frac{6p(T-1)\beta}{1-4p(T-1)\beta},
    \frac{12p\beta}{1-8p\beta},
    4(2p-1)+\frac{2}{\left(\sqrt{(4p)'}-1\right)^2}\Bigg\} >0.
\end{aligned}
\end{equation}

This specific choice of $\beta$ and $q^{*}(p, T, \beta)$ will become clear later in the proof. However let us record the following important implication of this choice: if $\beta$ satisfies \eqref{assume:parameters} and $\mu\in \sp((\R^d)^T)$ satisfies $\e_{q/(2\sigma^2)}(\mu)<\infty$ for some $q>q^{*}(p, T, \beta)$, then the parameters $(\beta, q)$ satisfy the assumptions of Proposition \ref{prop:lipkernel} and Lemma \ref{lem:dualnorm} for $\w_{2p}$ as follows:
\begin{align}
    0<\beta<\frac{1}{8p}<\frac{1}{2}< \frac{1}{(2p)'}, \quad 
    q>q^{*}(p, T, \beta)
    \ge \frac{2(2p-1)}{\beta}
    >2(2p-1)\left(\frac{1}{\beta}-(2p)'\right)
    >0.
\end{align}

For ease of reference, we restate Proposition~\ref{prop:lipkernel} and Lemma~\ref{lem:dualnorm} with $p$ replaced by $2p$.

\begingroup
\def\thethm{\ref{prop:lipkernel}}
    \begin{prop}[with $p$ replaced by $2p$] Let $\beta$ satisfy \eqref{assume:parameters}. Suppose that $\mu\in \sp((\R^d)^T)$ satisfies $\e_{q/(2\sigma^2)}(\mu)<\infty$ where $q>q^{*}(p, T,\beta)$. Then there exists a constant $C>0$ that depends only on $d, p, \sigma, q, \beta, \e_{q/(2\sigma^2)}(\mu)$ such that for all $x,y\in (\R^d)^T$ and $t\in \{1,2,\ldots, T-1\}$,
    \begin{align}
    \w_{2p}((\mu^{\sigma})_{x_{1:t}}, (\mu^{\sigma})_{y_{1:t}})
    \le C e^{\frac{\beta}{\sigma^2}(|x_{1:t}-\mean(\mu_{1:t})|\vee|y_{1:t}-\mean(\mu_{1:t})|)^2} \abs{x_{1:t}-y_{1:t}}.\label{eq:w_2p lipkernel}
\end{align}
\end{prop}
\addtocounter{thm}{-1}
\endgroup

\begingroup
\def\thethm{\ref{lem:dualnorm}}
\begin{lem}[with $p$ replaced by $2p$] Let $\beta$ satisfy \eqref{assume:parameters}. Suppose that $\mu\in \sp((\R^d)^T)$ satisfies $\e_{q/(2\sigma^2)}(\mu)<\infty$ where $q>q^{*}(p, T,\beta)$. If $\gamma\in \sp(\R^d)$ is absolutely continuous with respect to the Lebesgue measure, then there exists a constant $C>0$ that depends only on $d, p, \beta, \sigma, \e_{q/(2\sigma^2)}(\mu)$ such that for all $x\in (\R^d)^T$ and $t\in \{1,2,\ldots, T-1\}$,
\begin{align}
    \w_{2p}((\mu^{\sigma})_{x_{1:t}}, \gamma)
    \le C e^{\frac{\beta}{2\sigma^2}\abs{x_{1:t}-\mean(\mu_{1:t})}^2}\norm{(\mu^{\sigma})_{x_{1:t}}-\gamma}_{\f^{\sigma, 2p}}.\label{eq:w_2p dualnorm}
\end{align}
\end{lem}
\addtocounter{thm}{-1}
\endgroup

\begin{rmk}[Choice of 2p in Lemma \ref{lem:dppbound}]
It is worthwhile to emphasize that the $\w_{2p}$-distances in \eqref{eq:lem dppbound} can be replaced by $\w_r$-distances for any $r>p$ under suitable choices of parameters $\beta>0$ and $q^{*}(p, T, \beta)>0$. Our choice of $\beta$ and $q^{*}(p, T, \beta)$ in \eqref{assume:parameters} and \eqref{assume:q star} is tailored specifically to establish the upper bound \eqref{eq:lem dppbound} in terms of $\w_{2p}$-distances. Indeed, let us revisit \eqref{eq:W_2p vs W_r} and choose $\gamma$ as an optimal coupling for $\w_{r}((\mu^{\sigma})_1, (\mhat{\mu}^{\sigma}_n)_1)$ instead of $\w_{2p}((\mu^{\sigma})_1, (\mhat{\mu}^{\sigma}_n)_1)$. Following a similar argument we obtain
\begin{align}
    \a\w^{(\sigma)}_p(\mu, \mhat{\mu}_n)
    \le C \left(\e_{\vartheta}(\mhat{\mu}_n)\right)^{1/r}
    \left(\w^{(\sigma)}_{r}(\mu_1, (\mhat{\mu}_n)_1)
    +\sum_{t=1}^{T-1}\left(
    \int \w_{r}((\mu^{\sigma})_{y_{1:t}}, (\mhat{\mu}^{\sigma}_n)_{y_{1:t}})^{r} \mhat{\mu}^{\sigma}_n(dy)
    \right)^{1/r}\right)
\end{align}
for some $\vartheta>0$ that we will not specify here.
\end{rmk}

\begin{proof}[Proof of Lemma \ref{lem:dppbound}]
Set $m_{1:t}:=(\abs{x_{1:t}}\vee \abs{y_{1:t}})^2$ for $t\in \{1,\ldots, T-1\}$ and $m_{1:0}:=0$. Let $\gamma^{\star}_{x_{1:t}, y_{1:t}}$ be an optimal coupling of $\w_{2p}((\mu^{\sigma})_{x_{1:t}}, (\mhat{\mu}^{\sigma}_n)_{y_{1:t}})$ for $t\in \{0,1,\ldots, T-1\}$ with the convention that $(\mu^{\sigma})_{x_{1:0}}=(\mu^{\sigma})_1$ and $(\mhat{\mu}^{\sigma}_n)_{y_{1:0}}:=(\mhat{\mu}^{\sigma}_n)_1$. We define $E(x_{1:T}, y_{1:T}):=1$ and
\begin{align}
    E(x_{1:t}, y_{1:t})
    :=e^{\frac{p\beta}{\sigma^2}m_{1:t}}
    \norm{E(x_{1:t+1}, y_{1:t+1})}_{L^2(\gamma^{\star}_{x_{1:t},y_{1:t}})}, t\in \{0,1,\ldots, T-1\}.\label{eq:ind E}
\end{align}
Moreover, for each $t\in \{0,1,\ldots, T-1\}$, we define $W_{t}(y_{1:t}):=\w_{2p}((\mu^{\sigma})_{y_{1:t}}, (\mhat{\mu}^{\sigma}_n)_{y_{1:t}})^p$ and for all $s\in \{t+1, \ldots, T-1\}$, we set
\begin{align}
    W_s(y_{1:t})
    :=\norm{W_s(y_{1:t+1})}_{L^2((\mhat{\mu}^{\sigma}_n)_{y_{1:t}} )}.\label{eq:ind W}
\end{align}
For the value functions $V_t$ defined in Proposition \ref{prop:dpp}, we now show by the backward induction that for all $t\in \{0, 1, \ldots, T-1\}$,
\begin{align}
    V_{t}(x_{1:t}, y_{1:t})
    \le C E(x_{1:t}, y_{1:t})
    \left(\abs{x_{1:t}-y_{1:t}}^p
    +\sum_{s=t}^{T-1}
    W_{s}(y_{1:t})\right)\label{eq:indhyp}
\end{align}
for some constant $C>0$ with the convention $\abs{x_{1:0}-y_{1:0}}:=0$. Since $\mean(\mu)=0$, Proposition~\ref{prop:lipkernel} (in particular, the estimate \eqref{eq:w_2p lipkernel}) shows that there exists a finite constant $C>0$ that depends on $d, p, \sigma, q, \beta, \e_{q/(2\sigma^2)}(\mu)$ such that
\begin{align}
    \w_{2p}((\mu^{\sigma})_{x_{1:t}}, (\mu^{\sigma})_{y_{1:t}})^p
    \le C e^{\frac{p\beta}{\sigma^2}m_{1:t}}\abs{x_{1:t}-y_{1:t}}^p.\label{eq:dpplip}
\end{align}
 By the triangle inequality and \eqref{eq:dpplip}, we establish the initial case
\begin{align}
    V_{T-1}(x_{1:T-1}, y_{1:T-1})
    &\le \w_{2p}((\mu^{\sigma})_{x_{1:T-1}}, (\mhat{\mu}^{\sigma}_n)_{y_{1:T-1}})^p\\
    &\le  C e^{\frac{p\beta}{\sigma^2}m_{1:T-1}}\abs{x_{1:T-1}-y_{1:T-1}}^p
    +C\w_{2p}((\mu^{\sigma})_{y_{1:T-1}}, (\mhat{\mu}^{\sigma}_n)_{y_{1:T-1}})^p.
\end{align}
Now, let us assume that the induction hypothesis \eqref{eq:indhyp} holds for $V_{t+1}$. Using the shorthand notation $\norm{\cdot}_{L^2}:=\norm{\cdot}_{L^2(\gamma^{\star}_{x_{1:t}, y_{1:t}})}$,
the Cauchy--Schwarz inequality shows that
\begin{align}
    &V_{t}(x_{1:t}, y_{1:t})\\
    &\le \int \abs{x_{t+1}-y_{t+1}}^p +V_{t+1}(x_{1:t+1}, y_{1:t+1})\gamma^{\star}_{x_{1:t}, y_{1:t}}(dx_{t+1}, dy_{t+1})\\
    &\underbrace{\le}_{\eqref{eq:indhyp}} C \int E(x_{1:t+1}, y_{1:t+1})
    \left(\abs{x_{1:t}-y_{1:t}}^p
    +\abs{x_{t+1}-y_{t+1}}^p
    +\sum_{s=t+1}^{T-1}W_s(y_{1:t+1})\right)
    \gamma^{\star}_{x_{1:t}, y_{1:t}}(dx_{t+1}, dy_{t+1})\\
    &\le C\norm{E(x_{1:t+1}, y_{1:t+1})}_{L^2}
    \left(\abs{x_{1:t}-y_{1:t}}^p
    +\w_{2p}((\mu^{\sigma})_{x_{1:t}}, (\mhat{\mu}^{\sigma}_n)_{y_{1:t}})^p
    +\sum_{s=t+1}^{T-1}\norm{W_s(y_{1:t+1})}_{L^2}\right).
\end{align}
From \eqref{eq:dpplip} and the triangle inequality, we obtain
\begin{align}
    \w_{2p}((\mu^{\sigma})_{x_{1:t}}, (\mhat{\mu}^{\sigma}_n)_{y_{1:t}})^p
    \le C e^{\frac{p\beta}{\sigma^2}m_{1:t}}\abs{x_{1:t}-y_{1:t}}^p
    +C W_{t}(y_{1:t}).
\end{align}
Together with \eqref{eq:ind E} and \eqref{eq:ind W}, this shows \eqref{eq:indhyp}. Thus,
\begin{align}\label{eq:ind final}
    \a\w^{(\sigma)}_p(\mu, \mhat{\mu}_n)^p
    \le CE(x_{1:0}, y_{1:0})\sum_{t=0}^{T-1}W_t(y_{1:0}).
\end{align}
Note that for $\gamma^{\star}(dx_{1:T}, dy_{1:T}):=\prod_{t=0}^{T-1}\gamma^{\star}_{x_{1:t}, y_{1:t}}(dx_{t+1}, dy_{t+1})\in \cpl(\mu^{\sigma}, \mhat{\mu}^{\sigma}_n)$,
\begin{align}
    E(x_{1:0}, y_{1:0})
    &=\left(\int \prod_{t=1}^{T-1}e^{\frac{2p\beta}{\sigma^2}m_{1:t}}\gamma^{\star}(dx, dy)\right)^{1/2}\\
    &\le \left(\int e^{\frac{2p\beta}{\sigma^2}(T-1)m_{1:T-1}}\gamma^{\star}(dx, dy)\right)^{1/2}
    \le\left(\e_{\frac{2p\beta(T-1)}{\sigma^2}}(\mu^{\sigma})+\e_{\frac{2p\beta(T-1)}{\sigma^2}}(\mhat{\mu}_n^{\sigma})\right)^{1/2}.
\end{align}
For the last inequality, we use $e^{\frac{2p\beta}{\sigma^2}(T-1)m_{1:T-1}}\le e^{\frac{2p\beta}{\sigma^2}(T-1)\abs{x}^2}+e^{\frac{2p\beta}{\sigma^2}(T-1)\abs{y}^2}$. Set $\theta:=\frac{2p\beta(T-1)}{\sigma^2}$. Note from \eqref{assume:parameters} and \eqref{assume:q star} that 
\begin{align}
    \theta= \frac{2p(T-1)}{\sigma^2} \beta
    < \frac{2p(T-1)}{\sigma^2} \frac{1}{4p(T-1)}
    =\frac{1}{2\sigma^2}
\end{align}
and $\e_{\frac{\theta}{1-2\sigma^2 \theta}}(\mu)<\infty$, as
\begin{align}
    \frac{q}{2\sigma^2}
    >\frac{q^{*}(p, T, \beta)}{2\sigma^2}
    \ge \frac{1}{2\sigma^2}\frac{6p(T-1)\beta}{1-4p(T-1)\beta}
    > \frac{2p\beta (T-1)}{\sigma^2 (1- 4p (T-1)\beta)}
    =\frac{\theta}{1-2\sigma^2 \theta}.
\end{align}
Therefore Lemma \ref{lem:exp mu sig} (Appendix~\ref{appendix:technical lemmas}) shows that 
\begin{align}
    E(x_{1:0}, y_{1:0})
    \le C\Big(\e_{\frac{2p\beta(T-1)}{\sigma^2(1-4p\beta(T-1))}}(\mhat{\mu}_n)\Big)^{1/2}.
\end{align}
It is evident from the definition of $W_t$ that
\begin{align}
    W_0(y_{1:0})=\mathcal{W}^{(\sigma)}_{2p}(\mu_1, (\mhat{\mu}_n)_1)^p,
    \quad
    W_t(y_{1:0})= \left(\int \mathcal{W}_{2p}((\mu^{\sigma})_{y_{1:t}}, (\mhat{\mu}^{\sigma}_n)_{y_{1:t}})^{2p} \mhat{\mu}^{\sigma}_n(dy)\right)^{1/2}.
\end{align}
Taking the $1/p$-th power in \eqref{eq:ind final} completes the proof of the lemma.
\end{proof}

\subsection{Step 3: Completion of the proof via empirical process theory}\label{sec:emp}

In this section, we complete the proof of Theorem \ref{thm:fast_rate}. The main ingredient is the following lemma, which is a classical result from empirical process theory.

\begin{lem}[Lemma $8$ in \cite{nietert2021smooth}]\label{lem:emp}
    Let $\h\subseteq C^{m}(\R^N)$ be a function class, and let $m$ be a positive integer with $m>N/2$. Let $\{\b_j\}_{j=1}^{\infty}$ be a cover of $\R^N$ consisting of nonempty bounded convex sets with $\sup_j \diam(\b_j)<\infty$. Set $M_j=\sup_{f\in \h}\norm{f}_{C^{m}(\b_j)}$ with $\norm{f}_{C^{m}(\b_j)}=\max_{\abs{\beta}\le m}\sup_{z\in \text{int}(\b_j)}\abs{\p^{\beta}f(z)}$. Then
    \begin{align}
        \E[\sqrt{n}\norm{\mu-\mhat{\mu}_n}_{\h}]
        \le C \sum_{j=1}^{\infty}M_j\mu(\b_j)^{1/2}\label{eq:lem emp}
    \end{align}
    for some constant $C>0$ that depends on $N, m, \sup_j \diam(\b_j)$.
\end{lem}

The rest of the proof proceeds as follows. In Lemma~\ref{lem:dpp norm bound}, we derive an upper bound for $\E[\a\w^{(\sigma)}_p(\mu, \widehat{\mu}_n)]$, which is expressed in terms of expected supremum norms over suitable function classes $\h$. The resulting bound corresponds to the left-hand side of \eqref{eq:lem emp}. In the proof of Lemma~\ref{lem:h bound}, we apply Lemma~\ref{lem:emp} to control this bound by a sum involving the measures of suitable covers. Finally, we complete the proof of Theorem~\ref{thm:fast_rate} at the end of this section.

To that end, let us define partitions and associated function classes. For $t\in \{1,2,\ldots, T-1\}$, we consider a partition $\{E^{(t)}_{k}\}_{k=0}^{\infty}$ of $(\R^d)^t$ which is defined as follows:
\begin{align}
    E^{(t)}_0:=\{y_{1:t}\in (\R^d)^t : \abs{y_{1:t}}<1\},
    \quad E^{(t)}_k:=\{y_{1:t}\in (\R^d)^t : k\le \abs{y_{1:t}}<k+1\} \text{ for } k\ge 1.
\end{align}
Recall $\f^{\sigma, 2p}$ in \eqref{eq:fc_f}. For each $t\in \{1,2,\ldots, T-1\}$ and non-negative integer $k$, we define $\h^{t, \sigma, 2p}_k$ as the set of all functions $h:(\R^d)^{t+1}\to \R$ such that
\begin{align}
    &h(z_{1:t+1})=\varphi_{\sigma}(y_{1:t}-z_{1:t}) \text{ for some } y_{1:t}\in E^{(t)}_k,\\
    &\text{ or }h(z_{1:t+1})=\varphi_{\sigma}(y_{1:t}-z_{1:t})f\ast \varphi_{\sigma}(z_{t+1})  \text{ for some } f\in \f^{\sigma, 2p}, y_{1:t}\in E^{(t)}_k.
\end{align}

We now present Lemma~\ref{lem:dpp norm bound} below.

\begin{lem}\label{lem:dpp norm bound}
Let $\beta$ satisfy $\eqref{assume:parameters}$. Suppose that $\mu\in \sp((\R^d)^T)$ satisfies $\mean(\mu)=0$ and $\mathcal{E}_{q/(2\sigma^2)}(\mu)<\infty$, where $q>q^{*}(p, T, \beta)$. Then there exists a constant $C>0$ that depends only on $d, T, p, \sigma, q, \beta, \e_{q/(2\sigma^2)}(\mu)$, such that
\begin{align}
    \E[\a\w^{(\sigma)}_p(\mu, \mhat{\mu}_n)]
    \le C\left(\frac{1}{\sqrt{n}} + \sum_{t=1}^{T-1} \sum_{k=0}^{\infty}k^{\frac{dt-1}{4p}}
    e^{\frac{(k+2)^2}{2\sigma^2(4p)'}} \E\left[\vert\vert \mu_{1:t+1}-(\mhat{\mu}_n)_{1:t+1}\vert\vert_{\h^{t, \sigma, 2p}_{k}}\right]\right).
\end{align}
\end{lem}

The estimate in Lemma~\ref{lem:dpp norm bound} is obtained by taking expectations in \eqref{eq:lem dppbound} from Lemma~\ref{lem:dppbound}. The following classical result on sums of independent random variables will be used in the proof.

\begin{prop}[Corollary $4$ and Section $5$ in \cite{fuk1971probability}]\label{prop:sum of iid}
    Let $Y_1, Y_2, \ldots, $ be i.i.d random variables and define $\overline{Y}_n:=\frac{1}{n}\sum_{j=1}^n Y_j$. Suppose $\E[\abs{Y_1}^r]<\infty$ for some $r\ge 1$ and $x>0$.
    \begin{enumerate}[label=(\alph*)]
    \item If $1\le r\le 2$, then there exists a positive constant $C$ that depends only on $r$ such that
    \begin{align}
        \P(\abs{\overline{Y}_n-\E[Y_1]}>x)
        \le C \E[\abs{Y_1-\E[Y_1]}^r]n^{1-r}x^{-r}.
    \end{align}
    \item If $r\ge 2$, then there exist positive constants $c$ and $C$ that depends only on $r$ such that
    \begin{align}
        \P(\abs{\overline{Y}_n-\E[Y_1]}>x)
        \le C \E[\abs{Y_1-\E[Y_1]}^r]n^{1-r}x^{-r}+e^{-cnx^2/\E[\abs{Y_1-\E[Y_1]}^2]}.
    \end{align}
    \end{enumerate}
\end{prop}

\begin{proof}[Proof of Lemma~\ref{lem:dpp norm bound}]

Note that Lemma~\ref{lem:dppbound} can be rewritten as
\begin{align}
    \a\w^{(\sigma)}_p(\mu, \mhat{\mu}_n)
    \le C (\e_{q_1/(2\sigma^2)}(\mhat{\mu}_n))^{1/(2p)}
    \left(\w_{2p}^{(\sigma)}(\mu_1, (\mhat{\mu}_n)_1)+\sum_{t=1}^{T-1}\left(\sum_{k=0}^{\infty}\text{I}^{(t)}_k\right)^{1/(2p)}\right).
\end{align}
Here, we set $q_1:=\frac{4p(T-1)\beta}{1-4p(T-1)\beta}$ and define
\begin{align}
    \text{I}^{(t)}_{k}
    :=\int_{E^{(t)}_k}\w_{2p}((\mu^{\sigma})_{y_{1:t}}, (\mhat{\mu}^{\sigma}_n)_{y_{1:t}})^{2p} \mhat{\mu}^{\sigma}_n(dy)
\end{align}
for $t\in \{1,2,\ldots, T-1\}$ and a non-negative integer $k$.

Let us set $q_2:=\frac{8p\beta}{1-8p\beta}$ and define the set $\Omega_0$ via
\begin{align}
    \Omega_0
    :=&\{\abs{\mean(\mhat{\mu}_n)-\mean(\mu)}< 1\}
    \cap\{\abs{M_2(\mhat{\mu}_n)-M_2(\mu)}<1\}\\
    &\cap\{\vert\e_{q_1/(2\sigma^2)}(\mhat{\mu}_n)-\e_{q_1/(2\sigma^2)}(\mu)\vert<1\}
    \cap\{\vert\e_{q_2/(2\sigma^2)}(\mhat{\mu}_n)-\e_{q_2/(2\sigma^2)}(\mu)\vert<1\}.
\end{align}
As specified in \eqref{assume:q star}, $q^{*}(p, T, \beta)\ge \frac{3q_1}{2}$. Hence $q/q_1>\frac{3}{2}>1$ and Proposition \ref{prop:sum of iid} shows that
\begin{align}
    \P(\vert\e_{q_1/(2\sigma^2)}(\mhat{\mu}_n)-\e_{q_1/(2\sigma^2)}(\mu)\vert>1)
    \le Cn^{1-q/q_1}
\end{align}
for some constant $C>0$ that depends on $p, \beta, T, q, \sigma, \e_{q/(2\sigma^2)}(\mu)$. Similarly, we obtain from $q^{*}(p, T, \beta)\ge\frac{3q_2}{2}$ that
\begin{align}
    \P(\vert\e_{q_2/(2\sigma^2)}(\mhat{\mu}_n)-\e_{q_2/(2\sigma^2)}(\mu)\vert>1)
    \le Cn^{1-q/q_2}
\end{align}
for a constant $C>0$ that depends on $p, \beta, T, q, \sigma, \e_{q/(2\sigma^2)}(\mu)$. As a consequence, $\P(\Omega^{c}_0)\le C n^{-r/2}$ where $r:=2(q/(q_1\vee q_2)-1)>1$. By H\"older's inequality applied to $1/r'+1/r=1$, we find
\begin{equation}\label{eq:omega 0}
\begin{aligned}
    \E[\a\w^{(\sigma)}_p(\mu, \mhat{\mu}_n)\ind{\Omega^{c}_0}]
    &\le \E[(M_p(\mu^{\sigma})^{1/p}+M_p(\mhat{\mu}^{\sigma}_n)^{1/p})\ind{\Omega^{c}_0}]\\
    &\le (\E[(M_p(\mu^{\sigma})^{1/p}+M_p(\mhat{\mu}^{\sigma}_n)^{1/p})^{r'}])^{1/r'}
    \P(\Omega^{c}_0)^{1/r}
    \le Cn^{-1/2},
\end{aligned}
\end{equation}
recalling that $\E[M_p(\mhat{\mu}^{\sigma}_n)^{s}]<\infty$ for all $s>0$. Also, note from \eqref{assume:parameters} and \eqref{assume:q star} that 
\begin{align}
    q^{*}(p, T, \beta)
    >\frac{2(2p-1)}{\beta}
    > \frac{2(2p-1)}{1/(8p)}
    \ge 2(2p-1).
\end{align}
As illustrated in \ref{ov:fastrate} (see \cite{nietert2021smooth} for details), this implies the fast rate
\begin{align}
    \E[\w^{(\sigma)}_{2p}(\mu_1, (\mhat{\mu}_n)_1)]\le Cn^{-1/2}
\end{align}
for some constant $C$ that depends on $d, p, \sigma, \e_{q/(2\sigma^2)}(\mu)$. Combining these results with the inequality $(\sum_{k} \text{I}^{(t)}_{k})^{1/(2p)}\le \sum_{k}(\text{I}^{(t)}_k)^{1/(2p)}$,
\begin{align}
    \E[\a\w^{(\sigma)}_{p}(\mu, \mhat{\mu}_n)]
    \le Cn^{-1/2}+\E[\a\w^{(\sigma)}_p(\mu, \mhat{\mu}_n)\ind{\Omega_0}]
    \le Cn^{-1/2}+C\sum_{t=1}^{T-1}\sum_{k=0}^{\infty}\E\big[\big(\text{I}^{(t)}_{k}\big)^{1/(2p)} \ind{\Omega_0}\big].
\end{align}

Hence, it suffices to show that for $t\in \{1,2,\ldots, T-1\}$ and non-negative integer $k$ we have
\begin{align}\label{eq:dpp norm bound c1}
    \text{I}^{(t)}_{k}
    \le C k^{\frac{dt-1}{2}}
    e^{\frac{(4p-1)(k+2)^2}{4\sigma^2}}
    \vert\vert \mu_{1:t+1}-(\mhat{\mu}_n)_{1:t+1}\vert\vert^{2p}_{\h^{t, \sigma, 2p}_{k}} \text{ on } \Omega_0.
\end{align}
Now, let us show \eqref{eq:dpp norm bound c1}. From $\mean(\mu)=0$ and Lemma~\ref{lem:dualnorm} (in particular, the estimate \eqref{eq:w_2p dualnorm}),
\begin{align}
    \w_{2p}((\mu^{\sigma})_{y_{1:t}}, (\mhat{\mu}^{\sigma}_n)_{y_{1:t}})
    \le C e^{\frac{\beta}{2\sigma^2}\abs{y_{1:t}}^2}
    \vert\vert(\mu^{\sigma})_{y_{1:t}}- (\mhat{\mu}^{\sigma}_n)_{y_{1:t}} \vert\vert_{\f^{\sigma, 2p}} \text{ a.s. }
\end{align}
If $f\in \f^{\sigma, 2p}$, we have
\begin{align}
    &((\mu^{\sigma})_{y_{1:t}}- (\mhat{\mu}^{\sigma}_n)_{y_{1:t}})(f)\\
    &= (\mu^{\sigma})_{y_{1:t}}(f) - \frac{ \int \varphi_{\sigma}(y_{1:t}-z_{1:t}) f\ast\varphi_{\sigma}(z_{t+1}) \mhat{\mu}_n(dz)}{\varphi_{\sigma}\ast \mhat{\mu}_n(y_{1:t})}\\
    &=(\mu^{\sigma})_{y_{1:t}}(f)
    \frac{\int \varphi_{\sigma}(y_{1:t}-z_{1:t})(\mhat{\mu}_n-\mu)(dz)}{\varphi_{\sigma}\ast \mhat{\mu}_n(y_{1:t})}
    +\frac{\int \varphi_{\sigma}(y_{1:t}-z_{1:t})f\ast\varphi_{\sigma}(z_{t+1})(\mu-\mhat{\mu}_n)(dz)}{\varphi_{\sigma}\ast \mhat{\mu}_n(y_{1:t})}\\
    &\le (\mu^{\sigma})_{y_{1:t}}(\abs{f})
    \frac{\vert\vert \mu_{1:t+1}-(\mhat{\mu}_n)_{1:t+1}\vert\vert_{\h^{t, \sigma, 2p}_{k}}}{\varphi_{\sigma}\ast \mhat{\mu}_n(y_{1:t})}
    +\frac{\vert\vert  \mu_{1:t+1}-(\mhat{\mu}_n)_{1:t+1}\vert\vert_{\h^{t, \sigma, 2p}_{k}}}{\varphi_{\sigma}\ast \mhat{\mu}_n(y_{1:t})}\\
    &\le Ce^{\frac{\beta}{2\sigma^2}\abs{y_{1:t}}^2}
    \frac{\vert\vert  \mu_{1:t+1}-(\mhat{\mu}_n)_{1:t+1}\vert\vert_{\h^{t, \sigma, 2p}_{k}}}{\varphi_{\sigma}\ast \mhat{\mu}_n(y_{1:t})}
    +\frac{\vert\vert  \mu_{1:t+1}-(\mhat{\mu}_n)_{1:t+1}\vert\vert_{\h^{t, \sigma, 2p}_{k}}}{\varphi_{\sigma}\ast \mhat{\mu}_n(y_{1:t})}.
\end{align}
To obtain the last inequality, we divide both sides of the following estimate by $\varphi_{\sigma}\ast \mu(y_{1:t})$. 
\begin{align}
    \int \abs{f(y_{t+1})}\varphi_{\sigma}\ast \mu(y_{1:t+1})dy_{t+1}
    \le C \int \varphi_{\sigma}(y_{1:t}-z_{1:t}) e^{\frac{(2p-1)\abs{z_{t+1}}^2}{\sigma^2}}\mu(dz)
    \le C  e^{\frac{\beta}{2\sigma^2}\abs{y_{1:t}}^2}\varphi_{\sigma}\ast \mu(y_{1:t}).
\end{align}
This estimate comes from a straightforward application of Lemma~\ref{lem:gaussiancomputation}\ref{lem:gaussiancomputation_b} and Lemma~\ref{lem:lbofdensity}\ref{lem:lbofdensity_b} (Appendix~\ref{appendix:technical lemmas}): Lemma~\ref{lem:gaussiancomputation}\ref{lem:gaussiancomputation_b} shows the first inequality. To obtain the second inequality, we choose $a:=\beta$ and $h(z):=e^{(2p-1)\abs{z_{t+1}}^2/\sigma^2}$ in Lemma~\ref{lem:lbofdensity}\ref{lem:lbofdensity_b}. The moment assumption of Lemma~\ref{lem:lbofdensity}\ref{lem:lbofdensity_b} is satisfied as $q^{*}(p, T, \beta)\ge 2(2p-1)/\beta$.

Plugging these in and applying H\"older's inequality we conclude
\begin{align}
    \text{I}^{(t)}_k
    &\le C\int_{E^{(t)}_k} e^{\frac{2p\beta\abs{y_{1:t}}^2}{\sigma^2}} \frac{\vert\vert  \mu_{1:t+1}-(\mhat{\mu}_n)_{1:t+1}\vert\vert^{2p}_{\h^{t, \sigma, 2p}_{k}}}{(\varphi_{\sigma}\ast \mhat{\mu}_n(y_{1:t}))^{2p}}
    \mhat{\mu}^{\sigma}_n(dy)\\
    &\le C (\e_{4p\beta/\sigma^2}(\mhat{\mu}^{\sigma}_n))^{1/2}
    \left(\int_{E^{(t)}_k}\frac{\vert\vert  \mu_{1:t+1}-(\mhat{\mu}_n)_{1:t+1}\vert\vert^{4p}_{\h^{t, \sigma, 2p}_{k}}}{(\varphi_{\sigma}\ast \mhat{\mu}_n(y_{1:t}))^{4p-1}} dy\right)^{1/2}.
\end{align}
From $\beta<1/(8p)$ in \eqref{assume:parameters}, we have $\theta:=4p\beta/\sigma^2<1/(2\sigma^2)$. Thus, Lemma~\ref{lem:exp mu sig} (Appendix~\ref{appendix:technical lemmas}) yields
\begin{align}
    \e_{\theta}(\mhat{\mu}^{\sigma}_n)
    =(1-2\sigma^2 \theta)^{-dT/2}\e_{q_2/(2\sigma^2)}(\mhat{\mu}_n)
    \le (1-2\sigma^2 \theta)^{-dT/2}(1+\e_{q_2/(2\sigma^2)}(\mu))<\infty \text{ on } \Omega_0,
\end{align}
where we have used that 
\begin{align}
\frac{\theta}{1-2\sigma^2\theta} = \frac{4p\beta}{1-8p\beta}\frac{1}{\sigma^2} = \frac{8p\beta}{1-8p\beta} \frac{1}{2\sigma^2} =\frac{q_2}{2\sigma^2}.
\end{align}
Jensen's inequality with $\int \abs{y_{1:t}-z_{1:t} }^2\mhat{\mu}_n(dz)=\vert y_{1:t}-\mean((\mhat{\mu}_n)_{1:t}) \vert^2+\var((\mhat{\mu}_n)_{1:t})$ shows that
\begin{align}
    \varphi_{\sigma}\ast \mhat{\mu}_n(y_{1:t})
    &=(2\pi \sigma^2)^{-dt/2}\int e^{-\frac{\abs{y_{1:t}-z_{1:t}}^2}{2\sigma^2}}\mhat{\mu}_n(dz)\\
    &\ge (2\pi \sigma^2)^{-dt/2}e^{-\frac{1}{2\sigma^2}\int \abs{y_{1:t}-z_{1:t}}^2\mhat{\mu}_n(dz)}\\
    &\ge (2\pi \sigma^2)^{-dt/2}
    e^{-\frac{1}{2\sigma^2}\abs{y_{1:t}-\mean((\mhat{\mu}_n)_{1:t})}^2}
    e^{-\frac{1}{2\sigma^2}\var((\mhat{\mu}_n)_{1:t})}.
\end{align}
Since $\mean(\mu)=0$ we have $\vert \mean((\mhat{\mu}_n)_{1:t}) \vert \le 1$ on $\Omega_0$. Similarly we have $\var((\mhat{\mu}_n)_{1:t})\le M_2(\mhat{\mu}_n)\le 1+M_2(\mu)$. Using this and the triangle inequality, we obtain that on $\Omega_0$,
\begin{align}
    \frac{1}{\varphi_{\sigma}\ast \mhat{\mu}_n(y_{1:t})}
    \le (2\pi \sigma^2)^{dt/2}e^{\frac{1}{2\sigma^2}\abs{y_{1:t}-\mean((\mhat{\mu}_n)_{1:t})}^2}
    e^{\frac{1}{2\sigma^2}\var((\mhat{\mu}_n)_{1:t})}
    \le (2\pi \sigma^2)^{dt/2}e^{\frac{1}{2\sigma^2}(\abs{y_{1:t}}+1)^2}e^{\frac{1+M_2(\mu)}{2\sigma^2}}.
\end{align}
Using the fact that $\abs{y_{1:t}}\le k+1$ if $y_{1:t}\in E^{(t)}_k$,
\begin{align}
    \int_{E^{(t)}_k}\frac{\vert\vert  \mu_{1:t+1}-(\mhat{\mu}_n)_{1:t+1}\vert\vert^{4p}_{\h^{t, \sigma, 2p}_{k}}}{(\varphi_{\sigma}\ast \mhat{\mu}_n(y_{1:t}))^{4p-1}} dy
    &\le C\int_{E^{(t)}_k} e^{\frac{(4p-1)(\abs{y_{1:t}}+1)^2}{2\sigma^2}}\vert\vert  \mu_{1:t+1}-(\mhat{\mu}_n)_{1:t+1}\vert\vert^{4p}_{\h^{t, \sigma, 2p}_{k}} dy\\
    &\le C k^{dt-1}e^{\frac{(4p-1)(k+2)^2}{2\sigma^2}}\vert\vert  \mu_{1:t+1}-(\mhat{\mu}_n)_{1:t+1}\vert\vert^{4p}_{\h^{t, \sigma, 2p}_{k}} \text{ on } \Omega_0.
\end{align}
For the last inequality, we use that the Lebesgue measure of $E^{(t)}_k$ can be bounded above by $k^{dt-1}$ up to a constant factor. This establishes \eqref{eq:dpp norm bound c1}.
\end{proof}

To apply Lemma~\ref{lem:emp} to the upper bound in Lemma~\ref{lem:dpp norm bound}, 
we now construct suitable covers. Combining these coverings with Lemma~\ref{lem:emp} yields the estimate stated in Lemma~\ref{lem:h bound}.

For $t\in \{1,2,\ldots, T-1\}$, we define a cover $\{\b^{(t)}_j\}_{j=1}^{\infty}$ of $(\R^d)^t$ as follows: First, set
\begin{align}
    N^{(t)}_{0}:=\{0\},\quad
    N^{(t)}_k := \text{a maximal $1$-separated subset of } E^{(t)}_k \text{ for }k\ge 1.
\end{align}
Observe that $E^{(t)}_k\subseteq \bigcup_{y_{1:t}\in N^{(t)}_k}B(y_{1:t}, 1)$ where $B(y_{1:t}, 1)$ is a closed Euclidean ball of radius $1$ centered at $y_{1:t}$. The cover $\{\b^{(t)}_j\}_{j=1}^{\infty}$ is then defined by relabeling the collection $\bigcup_{k\ge 0}\{B(y_{1:t}, 1) : y_{1:t}\in N^{(t)}_k\}$. By \cite[Equation $(12)$]{nietert2021smooth},
\begin{align}\label{eq:bound on N}
    \big\vert N^{(t)}_k \big\vert\le (2(k+1)+1)^{dt}-(2k-1)^{dt}=
    (2k+3)^{dt}-(2k-1)^{dt} \le C k^{dt-1}
\end{align}
for a constant $C>0$ that depends on $d,t$.

\begin{lem}\label{lem:h bound} 
Let $\mu\in \sp((\R^d)^T)$, $t\in \{1,2,\ldots, T-1\}$ and $k\ge 0$. If $0<\delta<1$, then there exists a constant $C>0$ that depends only on $d,t, \sigma, p$, such that
\begin{align}
    &\E\left[\sqrt{n}\vert\vert \mu_{1:t+1}-(\mhat{\mu}_n)_{1:t+1} \vert\vert_{\h^{t, \sigma, 2p}_{k}}\right]\\
    &\le C\sum_{\ell=0}^{\infty}\Big(
    (k+1)^{m_{t,d}}e^{-\frac{(1-\delta)^2}{2\sigma^2}k^2}
    +\ind{\{\ell\ge \delta k -2\}}\Big)
    (\ell+1)^{m_{t,d}}e^{\frac{2p-1}{\sigma^2}(\ell+2)^2}
    \sum_{j\in C^{(t+1)}_{\ell}}
    \mu_{1:t+1}(\b^{(t+1)}_j)^{1/2}.
\end{align}
Here $m_{t,d}:=\lfloor d(t+1)/2 \rfloor+1$ and $C^{(t+1)}_{\ell}$ is the set of all $j$ for which the center of $\b^{(t+1)}_{j}$ is contained in $E^{(t+1)}_{\ell}$.
\end{lem}

\begin{proof}
We apply Lemma \ref{lem:emp} to $\h^{t, \sigma, 2p}_{k}\subseteq C^{m_{t,d}}((\R^d)^{t+1})$ and the cover $\{\b^{(t+1)}_j\}_{j=1}^{\infty}$ of $(\R^d)^{t+1}$ defined above. Denoting $M_j:=\sup_{h\in \h^{t, \sigma, 2p}_{k}}\norm{h}_{C^{m_{t,d}}(\b^{(t+1)}_j)}$ we find
\begin{align}
    \E\left[\vert\vert \mu_{1:t+1}-(\mhat{\mu}_n)_{1:t+1} \vert\vert_{\h^{t, \sigma, 2p}_{k}}\right]
    &\le C \sum_{j=1}^{\infty}M_j \mu_{1:t+1}(\b^{(t+1)}_j)^{1/2}\\
    &\le C \sum_{\ell=0}^{\infty} \sum_{j\in C^{(t+1)}_{\ell}} M_j \mu_{1:t+1}(\b^{(t+1)}_j)^{1/2}.\label{eq:sum h bound}
\end{align}
It suffices to show that if $j\in C^{(t+1)}_{\ell}$,
\begin{align}
    M_{j}\le C\Big(
    (k+1)^{m_{t,d}}e^{-\frac{(1-\delta)^2}{2\sigma^2}k^2}
    +\ind{\{\ell\ge \delta k -2\}}\Big)
    (\ell+1)^{m_{t,d}}e^{\frac{2p-1}{\sigma^2}(\ell+2)^2}\label{eq:bound M_j}
\end{align}
for a constant $C>0$ depending only on $d, t, \sigma, p$.

To see this, let $h\in \h^{t, \sigma, 2p}_{k}$. We consider two possible forms of $h$. For each case, we bound the derivatives $\p^{\alpha}h(z_{1:t+1})$ for multi-indices $\alpha$ satisfying $\abs{\alpha}\le m_{t,d}$.

\textbf{Case 1:} First, suppose $h(z_{1:t+1})=\varphi_{\sigma}(y_{1:t}-z_{1:t})$ for $y_{1:t}\in E^{(t)}_k$. Note that $\p^{\alpha}h(z_{1:t+1})= p_{\alpha}(y_{1:t}-z_{1:t})\varphi_{\sigma}(y_{1:t}-z_{1:t})$ for some polynomial $p_{\alpha}$ of degree at most $m_{t,d}$. In particular, we can find a constant $C>0$ that depends on $d, t, \sigma$ such that
\begin{align}
    \abs{\p^{\alpha}h(z_{1:t+1})}
    \le C(1+\abs{y_{1:t}-z_{1:t}})^{m_{t,d}}
    \varphi_{\sigma}(y_{1:t}-z_{1:t})
\end{align}
for all $|\alpha|\le m_{t,d}$. Set $0<\delta<1$. If $\abs{z_{1:t}}\le \delta k$, then from $k\le \abs{y_{1:t}}<k+1$ and the triangle inequality, we obtain
\begin{align}
    (1-\delta)k
    \le \abs{y_{1:t}}-\abs{z_{1:t}}
    \le \abs{y_{1:t}-z_{1:t}}
    \le \abs{y_{1:t}}+\abs{z_{1:t}}
    \le (k+1)+\delta k
    < 2k+1.
\end{align}
In particular, if $\abs{z_{1:t}}\le \delta k$,
\begin{align}
    (1+\abs{y_{1:t}-z_{1:t}})^{m_{t,d}}
    \varphi_{\sigma}(y_{1:t}-z_{1:t})
    \le (2k+2)^{m_{t,d}}(2\pi \sigma^2)^{-dt/2}e^{-\frac{(1-\delta)^2 k^2}{2\sigma^2}}.
\end{align}
Hence we can find a constant $C>0$ depending on $d, t, \sigma$ such that
\begin{align}
    \abs{\p^{\alpha} h(z_{1:t+1})}
    \le C\left((k+1)^{m_{t,d}}e^{-\frac{(1-\delta)^2 k^2}{2\sigma^2}} +\ind{\{\abs{z_{1:t}}>\delta k\}}\right).\label{eq:first h bound}
\end{align}

\textbf{Case 2:} Now, suppose $h(z_{1:t+1})=\varphi_{\sigma}(y_{1:t}-z_{1:t})f\ast \varphi_{\sigma}(z_{t+1})$ for $f\in \f^{\sigma, 2p}$ and $y_{1:t}\in E^{(t)}_k$. Similar to the previous case, there exists a constant $C>0$ that depends on $d, t, \sigma$ such that for all $\abs{\alpha}\le m_{t,d}$,
\begin{align}
    &\abs{\p^{\alpha}h(z_{1:t+1})}\\
    &\le C(1+\abs{y_{1:t}-z_{1:t}})^{m_{t,d}} \varphi_{\sigma}(y_{1:t}-z_{1:t})
    \int \abs{f(w_{t+1})}(1+\abs{z_{t+1}-w_{t+1}})^{m_{t,d}}\varphi_{\sigma}(z_{t+1}-w_{t+1}) dw_{t+1}.
\end{align}
Set $\tilde{\eta}:=\sqrt{1/(4p)'}$. As in the proof of Lemma \ref{lem:gaussiancomputation}\ref{lem:gaussiancomputation_a} (Appendix~\ref{appendix:technical lemmas}) we estimate
\begin{align}
    &\int \abs{f(w_{t+1})}(1+\abs{z_{t+1}-w_{t+1}})^{m_{t,d}}\varphi_{\sigma}(z_{t+1}-w_{t+1}) dw_{t+1}\\
    &= \int \abs{f(w_{t+1})} (1+\abs{z_{t+1}-w_{t+1}})^{m_{t,d}} \varphi_{\sigma}(z_{t+1}-w_{t+1}) \varphi^{-1}_{\sigma\tilde{\eta}}(z_{t+1}-w_{t+1}) d\mathcal{N}_{\sigma \tilde{\eta}}(dw_{t+1})\\
    &\le 
    \norm{f}_{L^{(2p)'}(\n_{\sigma\tilde{\eta}}; \R^d)}
    \left(\int \varphi^{1-2p}_{\sigma \tilde{\eta}}(w_{t+1})(1+\abs{z_{t+1}-w_{t+1}})^{2pm_{t,d}} \varphi^{2p}_{\sigma}(z_{t+1}-w_{t+1}) dw_{t+1}\right)^{1/(2p)}\\
    &\le 
    D_{2p, d, \sigma}
    \left(\int \varphi^{1-2p}_{\sigma \tilde{\eta}}(w_{t+1})(1+\abs{z_{t+1}-w_{t+1}})^{2pm_{t,d}} \varphi^{2p}_{\sigma}(z_{t+1}-w_{t+1}) dw_{t+1}\right)^{1/(2p)}. \label{eq:cm integral}
\end{align}
The first inequality follows from H\"older's inequality. The constant $D_{2p, d, \sigma}$ defined in Lemma~\ref{lem:gaussiancomputation}\ref{lem:gaussiancomputation_0_poincare} (Appendix~\ref{appendix:technical lemmas}) shows the second inequality. To compute the Riemann integral appearing in \eqref{eq:cm integral}, observe that
\begin{align}
    \varphi^{1-2p}_{\sigma \tilde{\eta}}(w_{t+1})\varphi^{2p}_{\sigma}(z_{t+1}-w_{t+1})
    =c_1 e^{\frac{2p(2p-1)}{\sigma^2}\abs{z_{t+1}}^2}e^{-c_2\abs{w_{t+1}-c_3 z_{t+1}}^2}
\end{align}
for some $c_1, c_2, c_3>0$ that depend only on $d, p, \sigma$. It follows by choosing $a=(2p-1)/(2\sigma^2 \tilde{\eta}^2), b=2p/(2\sigma^2)$ in Lemma~\ref{lem:product gaussian} (Appendix~\ref{appendix:technical lemmas}) and noting that $\frac{ab}{b-a}=\frac{2p(2p-1)}{\sigma^2}$. Using the bound
\begin{align}
    (1+\abs{z_{t+1}-w_{t+1}})^{2pm_{t,d}}
    &\le C\left((1+\abs{z_{t+1}})^{2pm_{t,d}}
    +(1+\abs{w_{t+1}-c_3 z_{t+1}})^{2pm_{t,d}}\right)\\
    &\le C(1+\abs{z_{t+1}})^{2pm_{t,d}}
    (1+\abs{w_{t+1}-c_3 z_{t+1}})^{2pm_{t,d}}
\end{align}
for some constant $C>0$ depending on $p, d, \sigma, t$, we establish
\begin{align}
    &\int\varphi^{1-2p}_{\sigma \tilde{\eta}}(w_{t+1})(1+\abs{z_{t+1}-w_{t+1}})^{2pm_{t,d}} \varphi^{2p}_{\sigma}(z_{t+1}-w_{t+1}) dw_{t+1}\\
    &\le C(1+\abs{z_{t+1}})^{2p m_{t,d}} e^{\frac{2p(2p-1)}{\sigma^2}\abs{z_{t+1}}^2}
    \int (1+\abs{w_{t+1}-c_3z_{t+1}})^{2pm_{t,d}}e^{-c_2\abs{w_{t+1}-c_3z_{t+1}}^2} dw_{t+1}\\
    &= C(1+\abs{z_{t+1}})^{2p m_{t,d}} e^{\frac{2p(2p-1)}{\sigma^2}\abs{z_{t+1}}^2}
    \int (1+\abs{w_{t+1}})^{2pm_{t,d}}e^{-c_2\abs{w_{t+1}}^2} dw_{t+1}
\end{align}
for possibly different constant $C>0$. Plugging this back into \eqref{eq:cm integral},
\begin{align}
    \int \abs{f(w_{t+1})}(1+\abs{z_{t+1}-w_{t+1}})^{m_{t,d}}\varphi_{\sigma}(z_{t+1}-w_{t+1}) dw_{t+1}
    \le C(1+\abs{z_{t+1}})^{m_{t,d}} e^{\frac{2p-1}{\sigma^2}\abs{z_{t+1}}^2}
\end{align}
follows. Hence from \eqref{eq:first h bound} we establish
\begin{align}
    \abs{\p^{\alpha}h(z_{1:t+1})}
    \le C\left((k+1)^{m_{t,d}}e^{-\frac{(1-\delta)^2 k^2}{2\sigma^2}} +\ind{\{\abs{z_{1:t}}>\delta k\}}\right)
    (1+\abs{z_{t+1}})^{m_{t,d}}e^{\frac{2p-1}{\sigma^2}\abs{z_{t+1}}^2}
\end{align}
up to a constant factor $C$ that depends on $p, d, \sigma, t$.

Finally, we combine the two cases to deduce the desired estimate~\eqref{eq:bound M_j}. Consider $\b^{(t+1)}_j$, whose center is contained in $E^{(t+1)}_{\ell}$. If $z_{1:t+1}\in \b^{(t+1)}_{j}$, $\abs{z_{1:t+1}}\le \ell+2$. Consequently, $\sup_{z_{1:t+1}\in \b^{(t+1)}_j}\ind{\{\abs{z_{1:t}}>\delta k\}}\le \ind{\{\ell\ge \delta k -2\}}$. This implies
\begin{align}
    M_j
    =\sup_{h\in \h^{t, \sigma, 2p}_{k}}\norm{h}_{C^{m_{t,d}}(\b^{t+1}_j)}
    \le C\left((k+1)^{m_{t,d}}e^{-\frac{(1-\delta)^2 k^2}{2\sigma^2}} 
    +\ind{\{\ell\ge \delta k -2\}}\right)
    (\ell+3)^{m_{t,d}}e^{\frac{(2p-1)(\ell+2)^2}{\sigma^2}}.
\end{align}
Using $(\ell+3)^{m_{t,d}}\le C(\ell+1)^{m_{t,d}}$, we obtain the desired estimate.
\end{proof}

\begin{proof}[Proof of Theorem \ref{thm:fast_rate}]
Let $\beta$ satisfy \eqref{assume:parameters}. We first prove that the fast rate holds for $\mu\in \sp((\R^d)^T)$ if $\e_{q/(2\sigma^2)}(\mu)<\infty$ with $q>q^{*}(p, T, \beta)$. Since $\a\w^{(\sigma)}_p$-distance is translation invariant, we may assume $\mean(\mu)=0$. From Lemma \ref{lem:dpp norm bound}, it suffices to show that for $t\in \{1,2,\ldots, T-1\}$, the sum
\begin{align}
    \text{S}
    :=\sum_{k=0}^{\infty} 
    k^{\frac{dt-1}{4p}}
    e^{\frac{(k+2)^2}{2\sigma^2(4p)'}}
    \E\left[\sqrt{n}\vert\vert \mu_{1:t+1}-(\mhat{\mu}_n)_{1:t+1} \vert\vert_{\h^{t, \sigma, 2p}_{k}}\right]
\end{align}
is finite.
Let $0<\delta<1$. It follows from Lemma \ref{lem:h bound} and Fubini's theorem that this sum is bounded from above by $\text{S}_1+\text{S}_2$ up to a constant factor depending on $d, t, \sigma, p$, where
\begin{align}
    &\text{S}_1
    :=\sum_{k=0}^{\infty}k^{\frac{dt-1}{4p}}
    e^{\frac{(k+2)^2}{2\sigma^2(4p)'}}
    (k+1)^{m_{t,d}}e^{-\frac{(1-\delta)^2}{2\sigma^2}k^2}
    \sum_{\ell=0}^{\infty}(\ell+1)^{m_{t,d}}e^{\frac{2p-1}{\sigma^2}(\ell+2)^2}
    \sum_{j\in C^{(t+1)}_{\ell}}
    \mu_{1:t+1}(\b^{(t+1)}_j)^{1/2},\\
    &\text{S}_2
    :=\sum_{\ell=0}^{\infty} (\ell+1)^{m_{t,d}}e^{\frac{2p-1}{\sigma^2}(\ell+2)^2}
    \sum_{k\le (\ell+2)/\delta}
    k^{\frac{dt-1}{4p}}
    e^{\frac{(k+2)^2}{2\sigma^2(4p)'}}
    \sum_{j\in C^{(t+1)}_{\ell}}
    \mu_{1:t+1}(\b^{(t+1)}_j)^{1/2}.
\end{align}
Since $\e_{q/(2\sigma^2)}(\mu)<\infty$, Markov's inequality and \eqref{eq:bound on N} show that
\begin{align}
    \sum_{j\in C^{(t+1)}_{\ell}}
    \mu_{1:t+1}(\b^{(t+1)}_j)^{1/2}
    \le \big\vert N^{(t+1)}_{\ell}\big\vert \mu(\{y: \abs{y}\ge \ell-1\})^{1/2}
    \le C \ell^{d(t+1)-1}e^{-\frac{q}{4\sigma^2}(\ell-1)^2}.
\end{align}
Comparing leading terms, the sum $\text{S}_1$ is finite if $(1-\delta)^2>1/(4p)'$ and $q>4(2p-1)$. From the bound
\begin{align}
    \sum_{k\le (\ell+2)/\delta}
    k^{\frac{dt-1}{4p}}
    e^{\frac{(k+2)^2}{2\sigma^2(4p)'}}
    \le (\lfloor(\ell+2)/\delta \rfloor +1)
    (\lfloor(\ell+2)/\delta \rfloor +1)^{\frac{dt-1}{4p}}\exp\left(\frac{(\lfloor(\ell+2)/\delta \rfloor +3)^2}{2\sigma^2(4p)'}\right)
\end{align}
we conclude that the sum $\text{S}_2$ is finite if $q>4(2p-1)+2/(\delta^2(4p)')$. Note that $(1-\delta)^2>1/(4p)'$ if and only if $2/(\delta^2(4p)')>2/(\sqrt{(4p)'}-1)^2$. Since
\begin{align}
    q>q^{*}(p, T, \beta)\ge 4(2p-1)+\frac{2}{(\sqrt{(4p)'}-1)^2},
\end{align}
we can choose a $\delta$ so that $\text{S}_1$ and $\text{S}_2$ are finite. This establishes the desired result.

Now let us choose $\beta=\frac{1}{4p(T+9)}$ in the parameter set \eqref{assume:parameters}. Since $T\ge 2$ and $p>1$, it is evident that the maximum between the first three terms in \eqref{assume:q star} is $8p(2p-1)(T+9)$. By the Cauchy--Schwarz inequality,
\begin{align}
    4(2p-1)+\frac{2}{(\sqrt{(4p)'}-1)^{2}}
    &=4(2p-1)+2(\sqrt{(4p)'}+1)^2(4p-1)^2\\
    &\le 4(2p-1)+4((4p)'+1)(4p-1)^2
    =4(2p-1)+4(8p-1)(4p-1).
\end{align}
The quadratic inequality $4(2p-1)+4(8p-1)(4p-1)\le 88p(2p-1)$ holds for all $p>1$. From $8p(2p-1)(T+9)\ge 88p(2p-1)$, we have $q^{*}(p, T, \beta)=8p(2p-1)(T+9)$. This proves Theorem \ref{thm:fast_rate}.
\end{proof}

\appendix

\section{Technical lemmas}\label{appendix:technical lemmas}

In this section, we collect several technical lemmas and estimates involving Gaussian kernels that are used throughout the paper.

\begin{lem}\label{lem:lbofdensity}
    Let $\mu\in \sp_2((\R^d)^T)$, $x\in (\R^d)^T$ and $t\in \{1,2,\ldots, T\}$.
    \begin{enumerate}[label=(\alph*)]
        \item\label{lem:lbofdensity_a} If $0<a_1<a_2$, then
        \begin{align}
            (\varphi^{a_2}_{\sigma}\ast \mu(x_{1:t}))^{1/a_2}
            \le e^{\frac{(1-a_1/a_2)}{2\sigma^2}\var(\mu_{1:t})}
            e^{\frac{(1-a_1/a_2)}{2\sigma^2}\abs{x_{1:t}-\mean(\mu_{1:t})}^2}
            (\varphi^{a_1}_{\sigma}\ast \mu(x_{1:t}))^{1/a_1}.
        \end{align}
        \item\label{lem:lbofdensity_b} If $0<a<1$ and $h\in L^{1/a}(\mu)$, then
        \begin{align}
            \int \varphi_{\sigma}(x_{1:t}-y_{1:t})\abs{h(y_{1:T})}\mu(dy)
            \le \norm{h}_{L^{1/a}(\mu)}
            e^{\frac{a}{2\sigma^2}\var(\mu_{1:t})}
            e^{\frac{a}{2\sigma^2}\abs{x_{1:t}-\mean(\mu_{1:t})}^2}
            \varphi_{\sigma}\ast \mu(x_{1:t}).
        \end{align}
    \end{enumerate}
\end{lem}

\begin{proof}
\ref{lem:lbofdensity_a}: From Jensen's inequality and $$\int \abs{x_{1:t}-y_{1:t}}^2 \mu(dy) = \vert x_{1:t}-\mean(\mu_{1:t})\vert^2 + \var(\mu_{1:t})$$ we obtain
\begin{align}
    \varphi^{a_1}_{\sigma}\ast \mu(x_{1:t})
    &=\int (2\pi \sigma^2)^{-a_1 dt/2}e^{-\frac{a_1 \abs{x_{1:t}-y_{1:t}}^2}{2\sigma^2}} \mu(dy)\\
    &\ge (2\pi \sigma^2)^{-a_1 dt/2}e^{-\frac{a_1 \int \abs{x_{1:t}-y_{1:t}}^2 \mu(dy)}{2\sigma^2}}
    = (2\pi \sigma^2)^{-a_1 dt/2}
    e^{-\frac{a_1 }{2\sigma^2}\var(\mu_{1:t})}
    e^{-\frac{a_1 }{2\sigma^2}\abs{x_{1:t}-\mean(\mu_{1:t})}^2}.
\end{align}
Thus, we obtain
\begin{align}
    (\varphi^{a_1}_{\sigma}\ast \mu(x_{1:t}))^{1/a_1}
    &=(\varphi^{a_1}_{\sigma}\ast \mu(x_{1:t}))^{1/a_2}(\varphi^{a_1}_{\sigma}\ast \mu(x_{1:t}))^{1/a_1-1/a_2}\\
    &\ge (\varphi^{a_1}_{\sigma}\ast \mu(x_{1:t}))^{1/a_2}
    \left((2\pi \sigma^2)^{-a_1 dt/2}
    e^{-\frac{a_1 }{2\sigma^2}\var(\mu_{1:t})}
    e^{-\frac{a_1 }{2\sigma^2}\abs{x_{1:t}-\mean(\mu_{1:t})}^2}\right)^{1/a_1-1/a_2}\\
    &=[(2\pi\sigma^2)^{-dt(a_2-a_1)/2}\varphi^{a_1}_{\sigma}\ast \mu(x_{1:t})]^{1/a_2}
    e^{-\frac{(1-a_1/a_2)}{2\sigma^2}\var(\mu_{1:t})}
    e^{-\frac{(1-a_1/a_2)}{2\sigma^2}\abs{x_{1:t}-\mean(\mu_{1:t})}^2}.
\end{align}
Since $(2\pi \sigma^2)^{dt/2}\varphi_{\sigma}\le 1$ on $(\R^d)^t$ by definition and $a_2>a_1$, we have
\begin{align}
    [(2\pi\sigma^2)^{-dt(a_2-a_1)/2}\varphi^{a_1}_{\sigma}\ast \mu(x_{1:t})]^{1/a_2}
    &= [(2\pi\sigma^2)^{dt/2} \varphi_\sigma)]^{a_1-a_2} (\varphi_\sigma^{a_2} \ast \mu(x_{1:t}))^{1/a_2}\\
    &\ge (\varphi^{a_2}_{\sigma}\ast \mu(x_{1:t}))^{1/a_2}.
\end{align}
This shows the desired result.

\ref{lem:lbofdensity_b}: Set $r:=1/a>1$ and use H\"older's inequality to see that
\begin{align}
    \int \varphi_{\sigma}(x_{1:t}-y_{1:t})\abs{h(y_{1:T})}\mu(dy_{1:T})
    \le \norm{h}_{L^{r}(\mu)}(\varphi^{r'}_{\sigma}\ast \mu(x_{1:t}))^{1/r'}.
\end{align}
Applying the previous part \ref{lem:lbofdensity_a} with $a_1=1$ and $a_2=r'>1$ we obtain
\begin{align}
    (\varphi^{r'}_{\sigma}\ast \mu(x_{1:t}))^{1/r'}
    \le e^{\frac{a}{2\sigma^2}\var(\mu_{1:t})}
    e^{\frac{a}{2\sigma^2}\abs{x_{1:t}-\mean(\mu_{1:t})}^2}
    \varphi_{\sigma}\ast \mu(x_{1:t}).
\end{align}
This ends the proof.
\end{proof}

\begin{lem}\label{lem:product gaussian} For any $0<a<b$, $e^{a\abs{y}^2-b\abs{x-y}^2}=e^{\frac{ab}{b-a}\abs{x}^2} e^{-(b-a)\abs{y-\frac{b}{b-a}x}^2}$. In particular,
\begin{align}
    \int_{\R^d} e^{a\abs{y}^2-b\abs{x-y}^2}dy
    =\left(\frac{\pi}{b-a}\right)^{d/2}e^{\frac{ab}{b-a}\abs{x}^2}.
\end{align}
\end{lem}
\begin{proof}
Note that $a\abs{y}^2-b\abs{x-y}^2 = -(b-a)\vert y-\frac{b}{b-a}x \vert^2+\frac{ab}{b-a}\abs{x}^2$. Hence,
\begin{align}
    \int_{\R^d} e^{a\abs{y}^2-b\abs{x-y}^2}dy
    =e^{\frac{ab}{b-a}\abs{x}^2} \int_{\R^d} e^{-(b-a)\abs{y-\frac{b}{b-a}x}^2}dy
    =e^{\frac{ab}{b-a}\abs{x}^2} \left(\frac{\pi}{b-a}\right)^{d/2}.
\end{align}
\end{proof}

Recall the function class $\f^{\sigma, p}$ and $\eta=\sqrt{1/(2p)'}$ defined in \eqref{eq:fc_f} in Section~\ref{sec:lipkernel}.

\begin{lem}\label{lem:gaussiancomputation}
Let $f\in \f^{\sigma, p}$. Then the following statements hold.
\begin{enumerate}[label=(\alph*)]
    \item\label{lem:gaussiancomputation_0_poincare} There exists a constant $D_{p, d, \sigma}>0$ that depends only on $p, d, \sigma$ such that
    \begin{align}\label{eq:poincareconstant}
        \norm{f}_{L^{p'}(\n_{\sigma \eta})}\le D_{p, d, \sigma}.
    \end{align}
    \item\label{lem:gaussiancomputation_a} If $x\in \R^d$, then for a constant $D_{p, d, \sigma}$ in \eqref{eq:poincareconstant},
    \begin{align}
        \abs{f}\ast \varphi_{\sigma}(x)
        \le D_{p, d, \sigma}(2^{1/p}/(2p)')^{d/2}
        e^{\frac{(p-1)\abs{x}^2}{\sigma^2}}.
    \end{align} 
    \item\label{lem:gaussiancomputation_b} If $\mu\in \sp((\R^d)^T)$, $x, y\in (\R^d)^T$ and $t\in \{1,2,\ldots, T-1\}$, then for a constant $D_{p, d, \sigma}$ in \eqref{eq:poincareconstant},
    \begin{align}
        \int \abs{f(x_{t+1})}\varphi_{\sigma}\ast \mu(x_{1:t+1}) dx_{t+1}
        \le D_{p, d, \sigma}(2^{1/p}/(2p)')^{d/2}
        \int \varphi_{\sigma}(x_{1:t}-y_{1:t})e^{\frac{(p-1)\abs{y_{t+1}}^2}{\sigma^2}} \mu(dy).
    \end{align} 
\end{enumerate}
\end{lem}
\begin{proof}
\ref{lem:gaussiancomputation_0_poincare}: Let us recall that Gaussian measures on $\R^d$ satisfy the $r$-Poincar\'e inequality \cite[Theorem 2.4]{milman2009role} for all $1\le r<\infty$. In particular, for any $\xi>0$ and $1\le r<\infty$, there exists a constant $D>0$ depending only on $d, r, \xi$ such that
\begin{align}
    \norm{\psi-\n_{\xi}(\psi)}_{L^r(\n_{\xi})}
    \le D \norm{\nab \psi}_{L^r (\n_{\xi}; \R^d)} \text{ for all } \psi\in C^{\infty}_c(\R^d).
\end{align}
By taking $\xi=\sigma \eta$ and $r=p'$, we find a constant $D_{p, d, \sigma}>0$ that depends only on $p, d, \sigma$ such that
\begin{align}
    \norm{\psi-\n_{\sigma\eta}(\psi)}_{L^{p'}(\n_{\sigma\eta})}
    \le D_{p, d, \sigma} \norm{\nab \psi}_{L^{p'} (\n_{\sigma\eta}; \R^d)} \text{ for all } \psi\in C^{\infty}_c(\R^d).
\end{align}
This show the desired estimate.

\ref{lem:gaussiancomputation_a}: We first apply H\"older's inequality and \eqref{eq:poincareconstant} to obtain
\begin{align}
    \abs{f}\ast \varphi_{\sigma}(x)
    =\int \abs{f(y)}\varphi_{\sigma\eta}^{-1}(y) \varphi_{\sigma}(x-y) \n_{\sigma\eta}(dy)
    &\le \norm{f}_{L^{p'}(\n_{\sigma\eta}; \R^d)}\left(\int \varphi^{1-p}_{\sigma\eta}(y)\varphi^p_{\sigma}(x-y) dy\right)^{1/p}\\
    &\le D_{p, d, \sigma}\left(\int \varphi^{1-p}_{\sigma\eta}(y)\varphi^p_{\sigma}(x-y) dy\right)^{1/p}.\label{eq:gaussiancomputationpf}
\end{align}
To conclude the proof of \ref{lem:gaussiancomputation_a}, it suffices to show that
\begin{align}
    \int \varphi_{\sigma\eta}^{1-p}(y)\varphi^p_{\sigma}(x-y)dy
    =2^{d/2}(1-1/(2p))^{dp/2}  e^{\frac{p(p-1)\abs{x}^2}{\sigma^2}}.
\end{align}
Note that
\begin{align}\label{eq:product gaussian 2}
    \varphi^{1-p}_{\sigma\eta}(y)\varphi^p_{\sigma}(x-y)
    =(2\pi \sigma^2\eta^2)^{d(p-1)/2}
    (2\pi \sigma^2)^{-dp/2}
    e^{a\abs{y}^2-b\abs{x-y}^2}
\end{align}
for $a:=(p-1)/(2\sigma^2\eta^2)$ and $b:=p/(2\sigma^2)$. We have $b>a$ as
\begin{align}
    \frac{p-1}{\eta^2} =\frac{p-1}{2p-1} 2p <p.
\end{align}
Hence, we apply Lemma~\ref{lem:product gaussian} to compute
\begin{align}
    \int \varphi_{\sigma\eta}^{1-p}(y)\varphi^p_{\sigma}(x-y)dy
    &=(2\pi \sigma^2\eta^2)^{d(p-1)/2}
    (2\pi \sigma^2)^{-dp/2} \int e^{a\abs{y}^2-b\abs{x-y}^2}dy\\
    &=(2\pi \sigma^2\eta^2)^{d(p-1)/2}
    (2\pi \sigma^2)^{-dp/2}(\pi /(b-a))^{d/2} e^{\frac{ab}{b-a}\abs{x}^2}.
\end{align}
This shows the desired result noting that 
\begin{align*}
b-a =\frac{p}{2\sigma^2} - \frac{p-1}{2\sigma^2 \eta^2} = \frac{p}{2\sigma^2} \Big( 1- \frac{2p-2}{2p-1}\Big) = \frac{p}{2\sigma^2} \frac{1}{2p-1} = \frac{1}{4\sigma^2 \eta^2},
\end{align*}
and 
\begin{align}\label{eq:b-a}
\frac{ab}{b-a} = \frac{(p-1) p}{(2\sigma^2)^2\eta^2} 4\sigma^2\eta^2 = \frac{(p-1)p}{\sigma^2}.
\end{align}

\ref{lem:gaussiancomputation_b}: By Fubini's theorem,
\begin{align}
    \int \abs{f(x_{t+1})}\varphi_{\sigma}\ast \mu(x_{1:t+1}) dx_{t+1}
    &= \int \int \abs{f(x_{t+1})} \varphi_{\sigma}(x_{1:t+1} -y_{1:t+1}) \mu(dy_{1:t+1})  dx_{t+1}\\
    &=\int \varphi_{\sigma}(x_{1:t}-y_{1:t})\abs{f}\ast \varphi_{\sigma}(y_{t+1})\mu(dy).
\end{align}
We now apply the previous result \ref{lem:gaussiancomputation_a} to bound $\abs{f}\ast \varphi_{\sigma}(y_{t+1})$. This concludes the proof.
\end{proof}

\begin{lem}\label{lem:exp mu sig}
    Let $0<\theta<1/(2\sigma^2)$ and $\mu\in \sp((\R^d)^T)$. Then $\e_{\theta}(\mu^{\sigma})=(1-2\sigma^2\theta)^{-dT/2}\e_{\frac{\theta}{1-2\sigma^2 \theta}}(\mu)$.
\end{lem}
\begin{proof}
By Fubini's theorem,
\begin{align}
    \e_{\theta}(\mu^{\sigma})
    =\int e^{\theta \abs{x}^2}\varphi_{\sigma}\ast \mu(x) dx
    =(2\pi \sigma^2)^{-dT/2}\int\left(\int_{(\R^d)^T} e^{\theta \abs{x}^2}e^{-\frac{\abs{x-y}^2}{2\sigma^2}}dx\right) \mu(dy).
\end{align}
Applying Lemma~\ref{lem:product gaussian} with $a=\theta$ and $b=1/(2\sigma^2)$ and computing
\begin{align}
    b-a
    =\frac{1}{2\sigma^2}-\theta, \quad    
    \frac{ab}{b-a}
    =\frac{\theta}{2\sigma^2} \frac{1}{1/(2\sigma^2)-\theta}
    =\frac{\theta}{1-2\sigma^2\theta}
\end{align}
yields
\begin{align}
    \int_{(\R^d)^T} e^{\theta \abs{x}^2}e^{-\frac{\abs{x-y}^2}{2\sigma^2}}dx
    =\left(\frac{\pi}{1/(2\sigma^2)-\theta}\right)^{dT/2}e^{\frac{\theta}{1-2\sigma^2 \theta}\abs{x}^2}.
\end{align}
This concludes the proof.
\end{proof}

\section{The fast rate is sharp}\label{appendix:sharp}

Let $a, b\in (\R^d)^T$ and $a\neq b$. Consider $\mu=\frac{1}{2}(\delta_a+\delta_b)$ and its empirical measure $\mhat{\mu}_n=\frac{1}{n}\sum_{j=1}^{n}\delta_{X^{(j)}}$ where $X^{(1)}, \ldots, X^{(n)}$ are i.i.d samples from $\mu$. In \cite[Remark after Lemma $5$]{nietert2021smooth}, it is shown that $\w_p^{(\sigma)}(\mhat{\mu}_n, \mu)$ has the lower bound
\begin{align}
    2^{-dT/2}\left(\E^{\mu}\left[e^{\frac{\abs{X}^2}{2\sigma^2}}\right]\right)^{-1}
    \sup\left\{(\mu^{\sigma}-\mhat{\mu}^{\sigma}_n)(\varphi) : \varphi\in C_c^{\infty}((\R^d)^T), \norm{\nab \varphi}_{L^{p'}(\n_{\sqrt{2}\sigma}; (\R^d)^T)} \le 1\right\}.\label{eq:lb}
\end{align}
Here, $\mu^{\sigma}:=\mu\ast \n_{\sigma}$ and $\mhat{\mu}^{\sigma}_n:=\mhat{\mu}_n\ast \n_{\sigma}$. Denoting by $Z_n=\frac{1}{n}\sum_{j=1}^n \ind{\{X^{(j)}=a\}}$, it is easy to compute that
\begin{align}
    \mu^{\sigma}
    = \frac{1}{2}\n(a, \sigma^2 \text{I}_{dT})
    +\frac{1}{2}\n(b, \sigma^2 \text{I}_{dT}),\quad
    \mhat{\mu}^{\sigma}_n
    =Z_n \n(a, \sigma^2 \text{I}_{dT})
    +(1-Z_n)\n(b, \sigma^2 \text{I}_{dT}).
\end{align}
A similar computation was used in \cite[Example (a) on page 2]{fournier2015rate}. The above implies
\begin{align}
    (\mu^{\sigma}-\mhat{\mu}^{\sigma}_n)(\varphi)
    =\int \varphi d\mu^{\sigma}-\int \varphi d\mhat{\mu}^{\sigma}_n
    =\left(1/2-Z_n\right)\left(\int \varphi d \n(a, \sigma^2 \text{I}_{dT})-\int \varphi d \n(b, \sigma^2 \text{I}_{dT})\right).
\end{align}
In particular, we obtain from \eqref{eq:lb} that
\begin{align}
    \w^{(\sigma)}_p(\mhat{\mu}_n, \mu)\ge C \abs{Z_n-1/2}
\end{align}
for some positive $C$. This concludes the proof since $\E[\abs{Z_n-1/2}]$ is of order $n^{-1/2}$.  

\begin{small}
\bibliographystyle{abbrv}
\bibliography{fast_aw2024}
\end{small}

\end{document}